\documentclass[12pt]{amsart}

\usepackage{amssymb, amsmath, xcolor}
\usepackage{import}
\usepackage{tikz}
\tikzset{filled/.style={minimum width=5pt,inner sep=0pt,circle,fill=black}}
\usetikzlibrary {arrows.meta}
\usepackage{subcaption}
\usepackage{mathtools}
\usepackage{todonotes}
\usepackage[utf8]{inputenc}
\usepackage{float}

\newtheorem{theorem}{Theorem}[section]
\newtheorem{lemma}[theorem]{Lemma}
\newtheorem{corollary}[theorem]{Corollary}
\newtheorem{conjecture}[theorem]{Conjecture}

\theoremstyle{definition}
\newtheorem{definition}[theorem]{Definition}
\newtheorem{example}[theorem]{Example}
\newtheorem{procedure}[theorem]{Procedure}

\theoremstyle{remark}
\newtheorem{remark}[theorem]{Remark}

\numberwithin{equation}{section}
\numberwithin{figure}{section}

\renewcommand{\mod}{\operatorname{mod}}
\newcommand{\Z}{\mathbb{Z}}
\DeclarePairedDelimiter\floor{\lfloor}{\rfloor}

\title[Cordial Trees and Friendship Graphs]{Klein Cordial Trees and Odd Cyclic Cordial Friendship Graphs}
\author[Erickson, Herden, Meddaugh, Sepanski, $\ldots$]{William Q. Erickson, Daniel Herden, Jonathan Meddaugh, Mark R. Sepanski, Isaac Echols, Cordell Hammon, Jorge Marchena-Menendez, Jasmin Mohn, Blanca Radillo-Murguia, Indalecio Ruiz-Bolanos}
\address{
All authors:
Department of Mathematics,
Baylor University,
1410 S.~4th Street,
Waco, TX 76706, USA}
\email{will\_erickson@baylor.edu, daniel\_herden@baylor.edu, jonathan\_meddaugh@baylor.edu,
mark\_sepanski@baylor.edu}
\date{\today}

\begin{document}

\begin{abstract}
    For a graph $G$ and an abelian group $A$, a labeling of the vertices of $G$ induces a labeling of the edges via the sum of adjacent vertex labels. Hovey introduced the notion of an $A$-cordial vertex labeling when both the vertex and edge labels are as evenly distributed as possible. Much work has since been done with trees, hypertrees, paths, cycles, ladders, prisms, hypercubes, and bipartite graphs. In this paper we show that
    all trees are $\mathbb{Z}_2^2$-cordial except for $P_4$ and $P_5$. In addition, we give numerous results relating to $\mathbb{Z}_m$-cordiality of the friendship graph $F_n$. The most general result shows that when $m$ is an odd multiple of $3$, then $F_n$ is $\mathbb{Z}_m$-cordial for all $n$. We also give a general conjecture to determine when $F_n$ is $\mathbb{Z}_m$-cordial.
\end{abstract}

\maketitle
\tableofcontents

\section{Introduction}

Many applications of graph theory reduce to questions of graph labeling. This has given rise to hundreds of labeling methods in the past 60 years \cite{gallian}. For a graph $G$ and an abelian group $A$, Hovey \cite{hovey} introduced the idea of an \emph{$A$-cordial labeling of $G$}.
These begin with a labeling of the vertices of $G$ by elements in $A$, and then
induce a labeling on the edges via the sum of adjacent vertex labels (see Definition~\ref{definition: induced labeling} below).
The labeling is said to be $A$-cordial if the distribution of vertex labels is as evenly distributed as possible and if the distribution of edge labels is as evenly distributed as possible (Definition \ref{definition: cordial}).
If $G$ admits an $A$-cordial labeling, we say that $G$ is $A$-cordial.
It has been shown that determining whether or not a graph admits an $A$-cordial labeling is NP-complete \cite{NPComplete}.

The $\Z_2$-cordiality of all trees was noted in \cite{Cahit}, while Hovey \cite{hovey} demonstrated that all trees also admit cordial labelings for  $\Z_3$, $\Z_4,$ and~$\Z_5$. These results were extended to $\Z_6$- and $\Z_7$-cordial labelings by Driscoll, Krop, and Nguyen \cite{treesAre6Cordial, treesAre7Cordial}.
Cordiality for hypertrees was studied by Cichacz, G\"{o}rlich, and Tuza \cite{CichaczGorlichTuza13}, Tuczyński, Wenus, and Wesek \cite{TuczynskiWenusKesek19}, and
 Cichacz and G\"{o}rlich \cite{CichaczGorlich21}.

Moving beyond trees, cordiality for paths and cycles was studied by Patrias and Pechenik \cite{patriasPechenik} and Cichacz \cite{Cichacz22}. Moving beyond cyclic groups,
$\mathbb{Z}_2^2$-cordiality has been studied for numerous graphs. For example, Riskin \cite{riskin} looked at $K_n$ and $K_{m,n}$, while Pechenik and Wise \cite{pechenikWise} looked at paths, cycles, ladders, prisms, and hypercubes.

In this paper, we give a complete answer to the ``seemingly difficult problem" \cite{riskin} of classifying all trees that admit a $\mathbb{Z}_2^2$-cordial labeling.
Precisely, Theorem \ref{theorem: main result on trees} shows that all trees, except $P_4$ and $P_5$, are $\mathbb{Z}_2^2$-cordial. We also conjecture (Conjecture \ref{conjecture: weakly cordial}) that trees are weakly $A$-cordial for any abelian group $A$, where \emph{weakly $A$-cordial} means that all but a finite number of trees are $A$-cordial.

Next we turn our attention to $\mathbb{Z}_m$-cordiality for friendship graphs. A \emph{friendship graph} $F_n$ is formed by joining $n$ copies of $C_3$ at a central vertex (Definition \ref{definition: friendship graph}). It has been shown that all friendship graphs are $\mathbb{Z}_4$-, $\mathbb{Z}_5$-, and $\mathbb{Z}_7$-cordial \cite{friendshipsAre4Cordial, friendshipsAre5Cordial, friendshipsAre7Cordial}.

In this paper, for $m$ odd, we first show that $F_n$ is $\mathbb{Z}_m$-cordial for all $n$ if and only if $F_1, F_2, \ldots, F_m$ are $\mathbb{Z}_m$-cordial (Theorem \ref{theorem: Fm or F2m reduction}).
When also ${3 \mid m}$, we use this to show that $F_n$ is $\mathbb{Z}_m$-cordial for all $n$ (Corollary~\ref{corollary done odd m div by 3}). In general, we show that $F_n$ is never $\mathbb{Z}_m$-cordial when
\[ 2\mid n, \,\, 4 \nmid n, \,\, m = \frac{3n}{d} \]
for $d$ a positive, odd divisor of $3n$ (Theorem \ref{theorem: obstruction}). In fact (Conjecture~\ref{conjecture for cordiality of friendship}), we conjecture that this is the only obstruction to $F_n$ being $\mathbb{Z}_m$-cordial.

A key idea for studying $\mathbb{Z}_m$-cordiality for friendship graphs is that of \emph{maximally additively disjoint} (\emph{MAD}) pairs in an abelian group $A$. These are
pairs $(a_i, b_i) \in A^2$, $1 \leq i \leq \mu$, such that
\[ a_1, b_1, \ldots, a_\mu, b_\mu, \, a_1 + b_1, \ldots, a_\mu + b_\mu \]
are all distinct. See \cite{Haanpaa04, Haanpaa07} for closely related ideas. In Theorem \ref{theorem: MAD}, we show that $\mathbb{Z}_m$ has exactly
$\floor*{\frac{m}{3}}$ MAD pairs when $m \not\equiv 6 \mod 12$ and $\floor*{\frac{m}{3}} - 1$ MAD pairs when $m \equiv 6 \mod 12$. As an application, this can be used to show that $F_1, F_2 \ldots,
F_{\floor*{m/3}}$ are $\mathbb{Z}_m$-cordial when $m \not\equiv 6 \mod 12$ (Corollary \ref{corollary up to m/3 via MAD}). In fact, for $m$ odd, this can be further pushed all the way up to $F_{\floor*{2m/3}- 1}$ or $F_{\floor*{2m/3}}$ (Theorems~\ref{theorem: to m/2} and \ref{theorem: F2m/3-1}), depending on a parity condition.

Finally, we observe that for a fixed graph $G$, there are at most a finite number of abelian groups $A$ for which $G$ is not $A$-cordial (Theorem~\ref{theorem: free cordial after}). This includes an explicit lower bound on $|A|$ that guarantees $A$-cordiality.

\section{Preliminaries}

Write $\mathbb{N}$ for $\mathbb{Z}_{\geq 0}.$
Throughout this paper, we will work with finite, simple, undirected graphs $G = (V, E)$, where $V=V(G)$ is the set of vertices and $E=E(G)$ is the set of edges. We write $A$ for an  abelian group. Given a labeling of the vertices of $G$ by elements of $A$, we will induce a labeling of the edges of $G$ by summing the corresponding vertex labels.

\begin{definition} \label{definition: induced labeling}
    Given a labeling $\ell:V \rightarrow A$ of the vertices of $G$ by elements of $A$, the \emph{induced labeling} $\ell:E \rightarrow A$ on the edges is given by
    \[ \ell((v_1, v_2)) = \ell(v_1) + \ell(v_2) \]
    for $(v_1, v_2) \in E$.
\end{definition}

For example, the vertex labeling of $P_3$ in Figure \ref{fig:labeled p3} gives rise to the displayed edge labels. When context allows, we will frequently refer to an edge or vertex by its label.

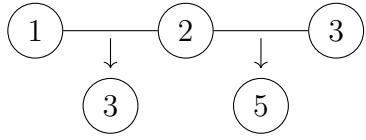
\begin{figure}[H]
    \centering
    \begin{tikzpicture}
    \tikzstyle{every node}=[draw,shape=circle];
    \path (0,0) node (1) {1};
    \path (2,0) node (2) {2};
    \path (4,0) node (3) {3};
    
    \draw (1)--(2)--(3);
    \draw[->] (1,-.1) -- (1,-.5);
    \draw[->] (3,-.1) -- (3,-.5);
    \path (1, -1) node (4) {3};
    \path (3, -1) node (5) {5};

\end{tikzpicture}
    \caption{A labeling of $P_3$ by elements of $\mathbb{Z}$}
    \label{fig:labeled p3}
\end{figure}

Given a labeling of the vertices of $G$ by an abelian group $A$, we will say the labeling is \emph{cordial} if the labels of $V$ are as evenly distributed as possible and if the induced labels of $E$ are as evenly distributed as possible.

\begin{definition} \label{definition: cordial}
    Let $A$ be an abelian group, $\ell:V \rightarrow A$ a labeling of the vertices of a graph $G$, and $\ell:E \rightarrow A$ the induced labeling of the edges.
    \begin{enumerate}
        \item
        Write $f_V:A\rightarrow \mathbb{N}$ and $f_E:A\rightarrow \mathbb{N}$ for the \emph{frequency distribution} of vertex and edge labels, respectively:
        \[ f_V(a) =  |\{v\in V \mid \ell(v) = a \}|  \]
        \[ f_E(a) =  |\{e\in E \mid \ell(e) = a \}| \]
        for $a \in A.$
        \item
        The labeling is called \emph{cordial} (or \emph{$A$-cordial}) if both
        \[ \max_{a_1, a_2 \in A} |f_V(a_1) - f_V(a_2)| \leq 1 \]
        \[ \max_{a_1, a_2 \in A} |f_E(a_1) - f_E(a_2)| \leq 1. \]
        \item
        $G$ is called $A$-\emph{cordial} if there exists an $A$-cordial labeling $\ell$ of~$G$.
    \end{enumerate}
\end{definition}

Note that if $|A| = m$ is finite, then letting $|V| = n_V$ and $|E| = n_E$, we can write
\[ n_V = q_V m + r_V \text{\,\, and \,\,} n_E = q_E m + r_E \]
per the Division Algorithm. Then the labeling is cordial if and only if (1) $f_V$ has a value of $q_V + 1$ for $r_V$ elements of $A$ and a value of $q_V$ for the rest of the elements of $A$, and (2) $f_E$ has a value of $q_E + 1$ for $r_E$ elements of $A$ and a value of $q_E$ for the rest of the elements of $A$.

Beginning in Section \ref{Tree Cordiality Set-Up}, we will study $\mathbb{Z}_2^2$-cordiality for trees. There we will use the following definition.

\begin{definition}
    We say a graph $H$ is an \emph{extension} of the graph $G$ (alternatively, $H$ \emph{extends} $G$) if $H$ contains $G$ as an induced subgraph. In this case we use $V'(H)$ and $E'(H)$ to denote vertices and edges in $H$ that are not in $G$.
\end{definition}

Beginning in Section \ref{section: obstruction to friendship cordiality}, we
will examine $\mathbb{Z}_m$-cordiality for the \emph{friendship graphs}.
It is a theorem of Erdős, Rényi, and Sós
\cite{ErdosRenyiSos}
that finite graphs in which every two vertices possess exactly one neighbor in common are exactly the friendship graphs.

\begin{definition} \label{definition: friendship graph}
    For $n \in \mathbb{N}$, the \emph{friendship graph} $F_n$ is the union of $n$ copies of a triangle $C_3$, joined at a central vertex.
\end{definition}

For example, Figure \ref{fig:F3} displays $F_3$.

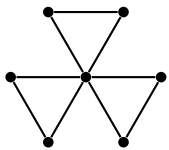
\begin{figure}[H]
    \centering
    \begin{tikzpicture}
    \node[filled, minimum width = 4pt] (0) at (0,0) {};
    \node[filled, minimum width = 4pt] (1) at (1*60: 1cm) {};
    \node[filled, minimum width = 4pt] (2) at (2*60: 1cm) {};
    \node[filled, minimum width = 4pt] (3) at (3*60: 1cm) {};
    \node[filled, minimum width = 4pt] (4) at (4*60: 1cm) {};
    \node[filled, minimum width = 4pt] (5) at (5*60: 1cm) {};
    \node[filled, minimum width = 4pt] (6) at (6*60: 1cm) {};
    
    \draw[black, thick] (0) -- (1);
    \draw[black, thick] (0) -- (2);
    \draw[black, thick] (1) -- (2);
    \draw[black, thick] (0) -- (3);
    \draw[black, thick] (0) -- (4);
    \draw[black, thick] (3) -- (4);
    \draw[black, thick] (0) -- (5);
    \draw[black, thick] (0) -- (6);
    \draw[black, thick] (5) -- (6);
\end{tikzpicture}
    \caption{The friendship graph $F_3$}
    \label{fig:F3}
\end{figure}

\section{Generalities and Finite Exceptions to Cordiality}

We begin with a straightforward observation, part of which can be found in \cite{hovey}.

\begin{lemma} \label{lemma: other cordial labelings}\mbox{}
\begin{itemize}
\item[$(a)$]
    Fix $s_0 \in A$. If $\ell: V \rightarrow A$ is an $A$-cordial labeling of $G$, then so is the labeling $\ell'$ defined by the \emph{label shifting}
    $\ell'(v) = \ell(v) - s_0$.
    Moreover, if $\varphi \in \operatorname{Aut}(A)$,
    then $\varphi \circ \ell$ is an $A$-cordial labeling
    of~$G$.
\item[$(b)$] Let $A$ be a ring with unity and let $u_0 \in A$ be a unit. If $\ell: V \rightarrow A$ is an $A$-cordial labeling of $G$, then so are the labelings $\ell'$ defined by $\ell'(v) = u_0\ell(v)$ and $\ell'(v) = \ell(v)u_0$, respectively.
\end{itemize}
\end{lemma}

Next, we show that a graph $G$ can only have, at most, a finite number of (finite) abelian groups for which it is not cordial. Write $\Delta(G)$ for the maximal degree of a vertex in $G$.

\begin{theorem} \label{theorem: free cordial after}
    Let $G$ be a graph and $A$ an abelian group. If
    \[ |A| > |V| + |E|\Delta(G) -1,    \]
    then $G$ is $A$-cordial.
\end{theorem}

\begin{proof}
    It suffices to show that, when $|A| > |V| + |E|\Delta(G) - 1$, there exists a labeling so that all edge and vertex labels have frequency $0$ or $1$. To do this, inductively label each vertex $v$ one at a time. At each stage, choose an element of $A$ that, first of all, is distinct from the previously chosen vertex labels. This excludes at most $|V| - 1$ elements of $A$. In addition, the vertex $v$ has at most $\Delta(G)$ adjacent edges that become labeled when $v$ is labeled.
    As the vertex labels are distinct, these new induced edge labels will be distinct amongst themselves. It remains to force no overlap with previous edge labels.
    In other words, if $(v, v_i)$ is such an edge, then the label for $v$ must also avoid any previously labeled edge minus $\ell(v_i)$. This excludes, at most, a further $|E|\Delta(G)$ elements of $A$.
\end{proof}

The above inequality can usually be dramatically improved by choosing a clever labeling order and keeping track of the exact number of resulting exclusions.

\section{Tree Cordiality Set-Up and Reductions} \label{Tree Cordiality Set-Up}

\begin{lemma}\label{lemma: extension of trees}
    Let $A$ be a finite abelian group with $|A| = n$, and let $T$ be an $A$-cordial tree with $nk$ vertices for some $k\in \mathbb{N}$. If $T^*$ is an extension of $T$ such that
    \[|V(T^*)| \leq nk +\floor*{\frac{n}{2}} + 1, \] then $T^*$ is also $A$-cordial.
\end{lemma}

\begin{proof}
    This proof proceeds similarly to the proof of \cite[Lemma 3]{hovey}; see
    also \cite[Lemma 2.1]{patriasPechenik}.
    Now, as the labeling of $T$ is cordial, each element in $A$ appears as a vertex label exactly $k$ times.
    Since there are $nk - 1$ edges, all but one element of $A$ appear as an edge label exactly $k$ times, while the remaining element, which we call $a_0$, appears as an edge label $k-1$ times.

    In general, if a single leaf is attached to $T$, then there is a unique vertex
    label that makes the extension cordial. Namely, the vertex label that forces the induced edge label to be $a_0$. At this point, the edge labels will be equally distributed, but one vertex label is used more than the others.

    We may now proceed inductively. If we attach a leaf to an $A$-cordial tree with $nk+j$ vertices, $1 \leq j < n$, there are $j$ vertex labels that must be avoided
    in order to preserve vertex-cordiality, leaving $n-j$ viable choices for the vertex label of the added leaf.
    To preserve edge-cordiality, there are $j - 1$ sums to be avoided for the vertex label of the added leaf.
    Both are always achievable when $n-j > j-1$. So when $j < \frac{n + 1}{2}$, the added leaf can be labeled to preserve cordiality. This will give the desired result.
\end{proof}

We focus on the Klein four-group, $\mathbb{Z}_2^2=\mathbb{Z}_2\times \mathbb{Z}_2$. We will write this group as  \[ \mathbb{Z}_2^2=\{0, a,b,c\},\]  with identity $0$, $a+a=b+b=c+c = 0$, $a+b=c$, $a+c=b,$ and $b+c=a.$
With this in mind, we restate more specific forms of Lemmas~\ref{lemma: other cordial labelings} and \ref{lemma: extension of trees}. The first follows from the observation that any bijection from $\mathbb{Z}_2^2$ to $\mathbb{Z}_2^2$ that maps $0$ to 0 must be an automorphism.

\begin{lemma} \label{lemma: bijection on V}
    If $\ell$ is a $\mathbb{Z}_2^2$-cordial labeling of $G$, then so is $\varphi\circ \ell$ for any bijection $\varphi: \mathbb{Z}_2^2 \to \mathbb{Z}_2^2$ with $\varphi(0) = 0.$
\end{lemma}

\begin{lemma}\label{extending 4k to 4k+3}
    Let $T$ be a $\mathbb{Z}_2^2$-cordial tree with $4k$ vertices for some $k\in \mathbb{N}$ and let $T^*$ be an extension of $T$ such that $|V(T^*)|\leq 4k+3.$ Then $T^*$ is $\mathbb{Z}_2^2$-cordial.
\end{lemma}

\section{$\mathbb{Z}_2^2$-Cordiality for Trees with $4k$ Vertices} \label{section: tree cord procedure}

In this section, we prove the following.

\begin{theorem}\label{extension of 4k vertices to 4k+4 verts}
    Let $T$ be a tree with $4k$ vertices for some $k\in \mathbb{N}.$ If $T$ is not $P_4$, then $T$ is $\mathbb{Z}_2^2$-cordial.
\end{theorem}

Before giving the proof, we introduce two procedures.

\begin{procedure}\label{proc I}
    Let $T$ be a tree with $4k$ vertices for some $k\in \mathbb{N}$ that is labeled in a $\mathbb{Z}_2^2$-cordial manner, and let $T^*$ be an extension of $T$ with $4k+4$ vertices in such a way that each vertex in $V'(T^*)$ is connected to a vertex in $V(T)$; see Figure \ref{fig: 4 leafs}, where the bottom vertices are part of $T$ and may coincide. This procedure
    constructs a $\mathbb{Z}_2^2$-cordial labeling on $T^*$.
      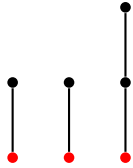
\begin{figure}[H]
        \begin{center}
          \begin{tikzpicture}
    \node[filled, red, minimum width = 4pt] (1) at (0,0) {};
    \node[filled, red, minimum width = 4pt] (3) at (0.75,0) {};
    \node[filled, red, minimum width = 4pt] (5) at (1.5,0) {};
    \node[filled, red, minimum width = 4pt] (7) at (2.25,0) {};
    \node[filled, minimum width = 4pt] (2) at (0,1) {};
    \node[filled, minimum width = 4pt] (4) at (0.75,1) {};
    \node[filled, minimum width = 4pt] (6) at (1.5,1) {};
    \node[filled, minimum width = 4pt] (8) at (2.25,1) {};
    
    \draw[black, thick] (1) -- (2);
    \draw[black, thick] (3) -- (4);
    \draw[black, thick] (5) -- (6);
    \draw[black, thick] (7) -- (8);
\end{tikzpicture}
        \end{center}
        \caption{Four added leaves of $T^*$, Procedure \ref{proc I}}
        \label{fig: 4 leafs}
     \end{figure}

    Notice that every $l\in \mathbb{Z}_2^2$ appears $k$ times as a vertex label of $T$. As there are $4k - 1$ edges, let $e\in \mathbb{Z}_2^2$ denote the unique label that is used only $k-1$ times as an edge label. All the remaining $l\in \mathbb{Z}_2^2$, $l\neq e$, appear exactly $k$ times as edge labels. If $|V(T)| = 0$ we choose $e=0$.

    Then there are $4$ edges in $E'(T^*)$ and each has one vertex in $V'(T^*)$ and one vertex in $V(T)$. These four vertices in $T$ can have the following labeling patterns: all vertices share the same label, three vertices share the same label and one vertex is labeled differently, two vertices share a label and the other two vertices share some other label, two vertices share the same label and the other two vertices have distinct labels, or each vertex is labeled distinctly.

    Thanks to  Lemmas \ref{lemma: other cordial labelings} and \ref{lemma: bijection on V}, we can relabel the vertices of $T$ in such a way that we have one of the cases found in Figure \ref{fig:labelings of 4 lines}.
        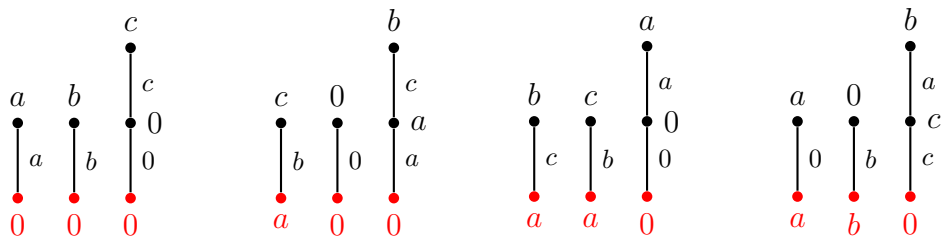
\begin{figure}[H]
        \begin{center}
              \begin{tikzpicture}
    \node[filled, red, minimum width = 4pt, label=below:\textcolor{red}{0}] (1) at (0,0) {};
    \node[filled, red, minimum width = 4pt, label=below:\textcolor{red}{0}] (3) at (1,0) {};
    \node[filled, red, minimum width = 4pt, label=below:\textcolor{red}{0}] (5) at (2,0) {};
    \node[filled, red, minimum width = 4pt, label=below:\textcolor{red}{0}] (7) at (3,0) {};
    \node[filled, minimum width = 4pt, label=above:0] (2) at (0,1) {};
    \node[filled, minimum width = 4pt, label=above:$a$] (4) at (1,1) {};
    \node[filled, minimum width = 4pt, label=above:$b$] (6) at (2,1) {};
    \node[filled, minimum width = 4pt, label=above:$c$] (8) at (3,1) {};
    
    \draw[black, thick] (1) -- (2) node [midway, right] {0};
    \draw[black, thick] (3) -- (4) node [midway, right] {$a$};
    \draw[black, thick] (5) -- (6) node [midway, right] {$b$};
    \draw[black, thick] (7) -- (8)node [midway, right] {$c$};

    \node[filled, red, minimum width = 4pt, label=below:\textcolor{red}{0}] (1) at (5,0) {};
    \node[filled, red, minimum width = 4pt, label=below:\textcolor{red}{0}] (3) at (6,0) {};
    \node[filled, red, minimum width = 4pt, label=below:\textcolor{red}{0}] (5) at (7,0) {};
    \node[filled, red, minimum width = 4pt, label=below:\textcolor{red}{$a$}] (7) at (8,0) {};
    \node[filled, minimum width = 4pt, label=above:$e$] (2) at (5,1) {};
    \node[filled, minimum width = 4pt, label=above:$b+e$] (4) at (6,1) {};
    \node[filled, minimum width = 4pt, label=above:$c+e$] (6) at (7,1) {};
    \node[filled, minimum width = 4pt, label=above:$a+e$] (8) at (8,1) {};
    
    \draw[black, thick] (1) -- (2) node [midway, right] {$e$};
    \draw[black, thick] (3) -- (4) node [midway, right] {{\footnotesize $b+e$}};
    \draw[black, thick] (5) -- (6) node [midway, right] {{\footnotesize $c+e$}};
    \draw[black, thick] (7) -- (8)node [midway, right] {$e$};

\end{tikzpicture}
          \begin{tikzpicture}

    \node[filled, red, minimum width = 4pt, label=below:\textcolor{red}{0}] (1) at (0,0) {};
    \node[filled, red, minimum width = 4pt, label=below:\textcolor{red}{0}] (3) at (1,0) {};
    \node[filled, red, minimum width = 4pt, label=below:\textcolor{red}{$a$}] (5) at (2,0) {};
    \node[filled, red, minimum width = 4pt, label=below:\textcolor{red}{$a$}] (7) at (3,0) {};
    \node[filled, minimum width = 4pt, label=above:$0$] (2) at (0,1) {};
    \node[filled, minimum width = 4pt, label=above:$a$] (4) at (1,1) {};
    \node[filled, minimum width = 4pt, label=above:$c$] (6) at (2,1) {};
    \node[filled, minimum width = 4pt, label=above:$b$] (8) at (3,1) {};
    
    \draw[black, thick] (1) -- (2) node [midway, right] {$0$};
    \draw[black, thick] (3) -- (4) node [midway, right] {$a$};
    \draw[black, thick] (5) -- (6) node [midway, right] {$b$};
    \draw[black, thick] (7) -- (8)node [midway, right] {$c$};

    \node[filled, red, minimum width = 4pt, label=below:\textcolor{red}{0}] (9) at (5,0) {};
    \node[filled, red, minimum width = 4pt, label=below:\textcolor{red}{0}] (11) at (6,0) {};
    \node[filled, red, minimum width = 4pt, label=below:\textcolor{red}{$a$}] (13) at (7,0) {};
    \node[filled, red, minimum width = 4pt, label=below:\textcolor{red}{$b$}] (15) at (8,0) {};
    \node[filled, minimum width = 4pt, label=above:$e$] (10) at (5,1) {};
    \node[filled, minimum width = 4pt, label=above:$b+e$] (12) at (6,1) {};
    \node[filled, minimum width = 4pt, label=above:$a+e$] (14) at (7,1) {};
    \node[filled, minimum width = 4pt, label=above:$c+e$] (16) at (8,1) {};
    
    \draw[black, thick] (9) -- (10) node [midway, right] {$e$};
    \draw[black, thick] (11) -- (12) node [midway, right] {{\footnotesize $b+e$}};
    \draw[black, thick] (13) -- (14) node [midway, right] {$e$};
    \draw[black, thick] (15) -- (16)node [midway, right] {{\footnotesize $a+e$}};

\end{tikzpicture}
          \begin{tikzpicture}
        \node[filled, red, minimum width = 4pt, label=below:\textcolor{red}{0}] (1) at (0,0) {};
    \node[filled, red, minimum width = 4pt, label=below:\textcolor{red}{$a$}] (3) at (1,0) {};
    \node[filled, red, minimum width = 4pt, label=below:\textcolor{red}{$b$}] (5) at (2,0) {};
    \node[filled, red, minimum width = 4pt, label=below:\textcolor{red}{$c$}] (7) at (3,0) {};
    \node[filled, minimum width = 4pt, label=above:$a$] (2) at (0,1) {};
    \node[filled, minimum width = 4pt, label=above:$c$] (4) at (1,1) {};
    \node[filled, minimum width = 4pt, label=above:$b$] (6) at (2,1) {};
    \node[filled, minimum width = 4pt, label=above:$0$] (8) at (3,1) {};
    
    \draw[black, thick] (1) -- (2) node [midway, right] {$a$};
    \draw[black, thick] (3) -- (4) node [midway, right] {$b$};
    \draw[black, thick] (5) -- (6) node [midway, right] {$0$};
    \draw[black, thick] (7) -- (8)node [midway, right] {$c$};
\end{tikzpicture}
        \end{center}
        \caption{Cordial leaves in Procedure \ref{proc I}}
        \label{fig:labelings of 4 lines}
     \end{figure}
    The indicated labelings ensure that $T^*$ is $\mathbb{Z}_2^2$-cordial.
\end{procedure}

\begin{procedure}\label{proc II}
    Let $T$ be a tree with $4k$ vertices for some $k\in \mathbb{N}$ that is labeled in a $\mathbb{Z}_2^2$-cordial manner, and let $T^*$ be an extension of $T$ with $4k+4$ vertices in such a way that there is one vertex in $V'(T^*)$ that is not connected to a vertex in $V(T)$; see Figure \ref{fig:unlabeled three lines}, where the bottom vertices are part of $T$ and need not be distinct.
    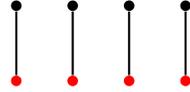
\begin{figure}[H]
        \begin{center}
          \begin{tikzpicture}
    \node[filled, red, minimum width = 4pt] (1) at (0,0) {};
    \node[filled, red, minimum width = 4pt] (3) at (0.75,0) {};
    \node[filled, red, minimum width = 4pt] (5) at (1.5,0) {};
    \node[filled, minimum width = 4pt] (2) at (0,1) {};
    \node[filled, minimum width = 4pt] (4) at (0.75,1) {};
    \node[filled, minimum width = 4pt] (6) at (1.5,1) {};
    \node[filled, minimum width = 4pt] (7) at (1.5,2) {};
    
    \draw[black, thick] (1) -- (2);
    \draw[black, thick] (3) -- (4);
    \draw[black, thick] (5) -- (6);
    \draw[black, thick] (6) -- (7);
\end{tikzpicture}
        \end{center}
        \caption{Four added leaves of $T^*$, Procedure \ref{proc II} }
        \label{fig:unlabeled three lines}
     \end{figure}

    As a result, $E'(T^*)$ contains four edges, and those edges are connected to exactly three vertices from $V(T)$. Again, thanks to Lemmas~\ref{lemma: other cordial labelings} and \ref{lemma: bijection on V}, there are only four cases to consider. Each of these cases is pictured in Figure \ref{fig:three lines with labels}.
    Moreover, the indicated labelings show that $T^*$ is $\mathbb{Z}_2^2$-cordial.
    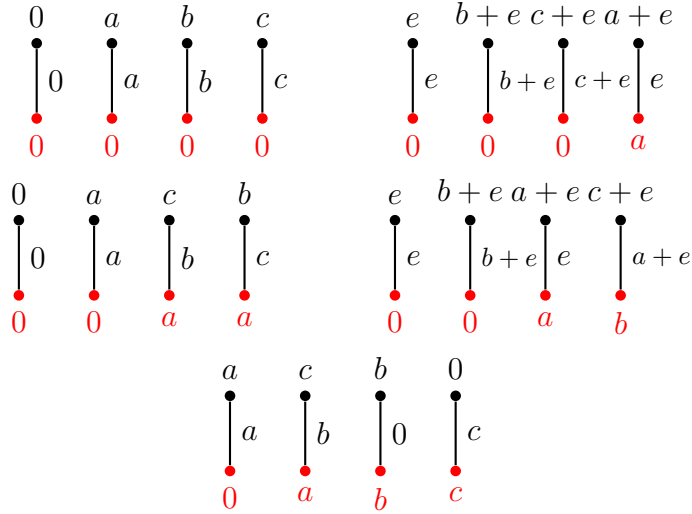
\begin{figure}[H]
        \begin{center}
          \begin{tikzpicture}
    \node[filled, red, minimum width = 4pt, label=below:\textcolor{red}{0}] (1) at (0,0) {};
    \node[filled, red, minimum width = 4pt, label=below:\textcolor{red}{$0$}] (3) at (0.75,0) {};
    \node[filled, red, minimum width = 4pt, label=below:\textcolor{red}{$0$}] (5) at (1.5,0) {};
    \node[filled, minimum width = 4pt, label=above:$a$] (2) at (0,1) {};
    \node[filled, minimum width = 4pt, label=above:$b$] (4) at (0.75,1) {};
    \node[filled, minimum width = 4pt, label=right:$0$] (6) at (1.5,1) {};
    \node[filled, minimum width = 4pt, label=above:$c$] (7) at (1.5,2) {};
    
    \draw[black, thick] (1) -- (2) node [midway, right] {{\footnotesize $a$}};
    \draw[black, thick] (3) -- (4) node [midway, right] {{\footnotesize $b$}};
    \draw[black, thick] (5) -- (6) node [midway, right] {{\footnotesize $0$}};
    \draw[black, thick] (6) -- (7)node [midway, right] {{\footnotesize $c$}};

    \node[filled, red, minimum width = 4pt, label=below:\textcolor{red}{$a$}] (1) at (3.5,0) {};
    \node[filled, red, minimum width = 4pt, label=below:\textcolor{red}{$0$}] (3) at (4.25,0) {};
    \node[filled, red, minimum width = 4pt, label=below:\textcolor{red}{$0$}] (5) at (5,0) {};
    \node[filled, minimum width = 4pt, label=above:$c$] (2) at (3.5,1) {};
    \node[filled, minimum width = 4pt, label=above:$0$] (4) at (4.25,1) {};
    \node[filled, minimum width = 4pt, label=right:$a$] (6) at (5,1) {};
    \node[filled, minimum width = 4pt, label=above:$b$] (7) at (5,2) {};
    
    \draw[black, thick] (1) -- (2) node [midway, right] {{\footnotesize $b$}};
    \draw[black, thick] (3) -- (4) node [midway, right] {{\footnotesize $0$}};
    \draw[black, thick] (5) -- (6) node [midway, right] {{\footnotesize $a$}};
    \draw[black, thick] (6) -- (7)node [midway, right] {{\footnotesize $c$}};
\end{tikzpicture}
          \hfill
          \begin{tikzpicture}
    \node[filled, red, minimum width = 4pt, label=below:\textcolor{red}{$a$}] (1) at (0,0) {};
    \node[filled, red, minimum width = 4pt, label=below:\textcolor{red}{$a$}] (3) at (.75,0) {};
    \node[filled, red, minimum width = 4pt, label=below:\textcolor{red}{$0$}] (5) at (1.5,0) {};
    \node[filled, minimum width = 4pt, label=above:$b$] (2) at (0,1) {};
    \node[filled, minimum width = 4pt, label=above:$c$] (4) at (0.75,1) {};
    \node[filled, minimum width = 4pt, label=right:$0$] (6) at (1.5,1) {};
    \node[filled, minimum width = 4pt, label=above:$a$] (7) at (1.5,2) {};
    
    \draw[black, thick] (1) -- (2) node [midway, right] {{\footnotesize $c$}};
    \draw[black, thick] (3) -- (4) node [midway, right] {{\footnotesize $b$}};
    \draw[black, thick] (5) -- (6) node [midway, right] {{\footnotesize $0$}};
    \draw[black, thick] (6) -- (7)node [midway, right] {{\footnotesize $a$}};

    \node[filled, red, minimum width = 4pt, label=below:\textcolor{red}{$a$}] (1) at (3.5,0) {};
    \node[filled, red, minimum width = 4pt, label=below:\textcolor{red}{$b$}] (3) at (4.25,0) {};
    \node[filled, red, minimum width = 4pt, label=below:\textcolor{red}{$0$}] (5) at (5,0) {};
    \node[filled, minimum width = 4pt, label=above:$a$] (2) at (3.5,1) {};
    \node[filled, minimum width = 4pt, label=above:$0$] (4) at (4.25,1) {};
    \node[filled, minimum width = 4pt, label=right:$c$] (6) at (5,1) {};
    \node[filled, minimum width = 4pt, label=above:$b$] (7) at (5,2) {};
    
    \draw[black, thick] (1) -- (2) node [midway, right] {{\footnotesize $0$}};
    \draw[black, thick] (3) -- (4) node [midway, right] {{\footnotesize $b$}};
    \draw[black, thick] (5) -- (6) node [midway, right] {{\footnotesize $c$}};
    \draw[black, thick] (6) -- (7)node [midway, right] {{\footnotesize $a$}};
\end{tikzpicture}
        \end{center}
        \caption{Cordial leaves in Procedure \ref{proc II}}
        \label{fig:three lines with labels}
     \end{figure}

\end{procedure}

We are now ready to prove Theorem \ref{extension of 4k vertices to 4k+4 verts}.
\begin{proof}
    We proceed by induction on $k$. We consider three base cases. If $k=0$, then the only tree with $4k$ vertices is the empty tree which is vacuously $\mathbb{Z}_2^2$-cordial.

    If $k=1$, then there are exactly two trees with $4k$ vertices, namely $P_4$ and $S_3$. $P_4$ is not cordial as shown in \cite[Theorem 3.4]{pechenikWise}. $S_3$ admits a $\mathbb{Z}_2^2$-cordial labeling as seen here:
    \begin{figure}[H]
        \begin{center}
          \begin{tikzpicture}
    \node[filled, minimum width = 4pt, label=north west:$0$] (1) at (0,0) {};
    \node[filled, minimum width = 4pt,label=below:$b$] (2) at (-.866,-0.5) {};
    \node[filled, minimum width = 4pt, label=above:$a$] (3) at (0,1) {};
    \node[filled, minimum width = 4pt, label=below:$c$] (4) at (.866,-.5) {};
    \node[filled, minimum width = 4pt] (5) at (4,0) {};
    \node[filled, minimum width = 4pt] (6) at (5,0) {};
    \node[filled, minimum width = 4pt] (7) at (6,0) {};
    \node[filled, minimum width = 4pt] (8) at (7,0) {};
    
    \draw[black, thick] (1) -- (2) node [midway, right] {{\footnotesize $b$}};
    \draw[black, thick] (1) -- (3) node [midway, right] {{\footnotesize $a$}};
    \draw[black, thick] (1) -- (4) node [midway, right] {{\footnotesize $c$}};
    \draw[black, thick] (5) -- (6);
    \draw[black, thick] (6) -- (7);
    \draw[black, thick] (7) -- (8);
\end{tikzpicture}
        \end{center}
        \caption{$S_3$ and $P_4$}
        \label{fig:S3 and P4}
     \end{figure}
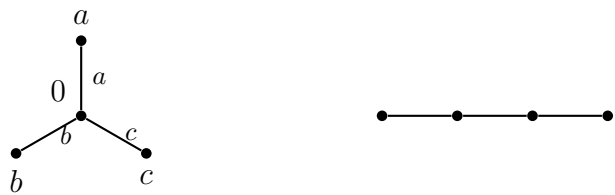

    If $k=2$, then there are $23$ trees with $4k$ vertices, see \cite[p.233]{harary}. The $10$ trees in Figure \ref{fig:Ten Trees} result from applying Procedure \ref{proc I} to $S_3$ and thus admit a $\mathbb{Z}_2^2$-cordial labeling.
    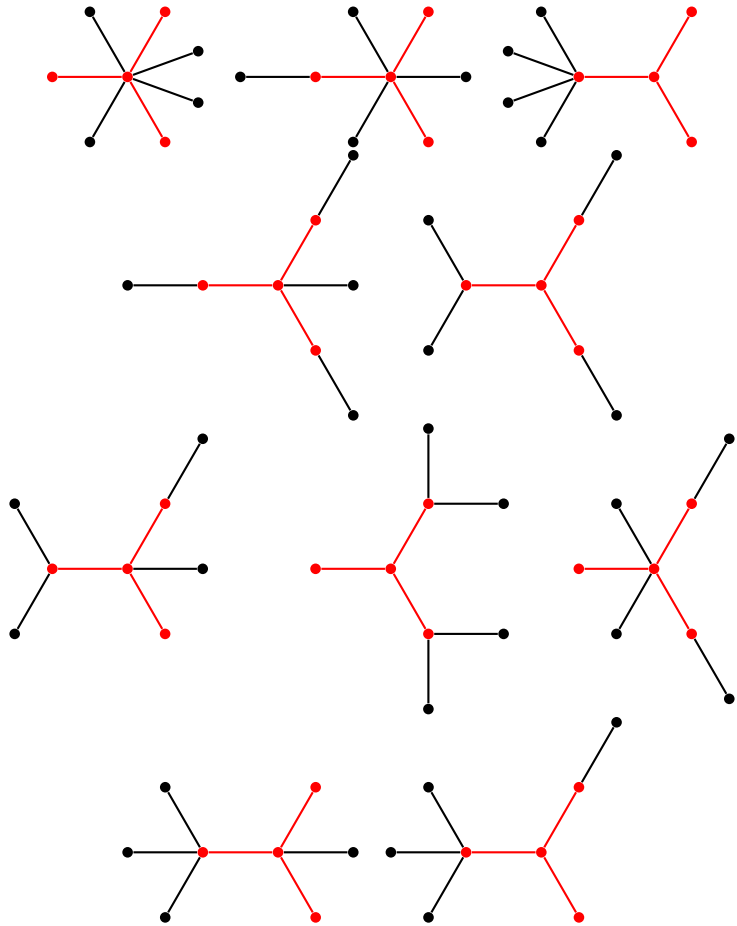
\begin{figure}[H]
        \begin{center}
          \begin{tikzpicture}

    \node[filled, red, minimum width = 4pt] (1) at (0-3.5,0) {};
    \node[filled, red, minimum width = 4pt] (2) at (-1-3.5,0) {};
    \node[filled, red, minimum width = 4pt] (3) at (.5-3.5,.866) {};
    \node[filled, red, minimum width = 4pt] (4) at (.5-3.5,-.866) {};
    \node[filled, minimum width = 4pt] (5) at (-.5-3.5,-.866) {};
    \node[filled, minimum width = 4pt] (6) at (-.5-3.5,.866) {};
    \node[filled, minimum width = 4pt] (7) at (.94-3.5,+.342) {};
    \node[filled, minimum width = 4pt] (8) at (.94-3.5,-.342) {};
    \draw[red, thick] (1) -- (2);
    \draw[red, thick] (1) -- (3);
    \draw[red, thick] (1) -- (4);
    \draw[black, thick] (1) -- (5);
    \draw[black, thick] (1) -- (6);
    \draw[black, thick] (1) -- (7);
    \draw[black, thick] (1) -- (8);

    \node[filled, red, minimum width = 4pt] (11) at (0,0) {};
    \node[filled, red, minimum width = 4pt] (12) at (-1,0) {};
    \node[filled, red, minimum width = 4pt] (13) at (.5,.866) {};
    \node[filled, red, minimum width = 4pt] (14) at (.5,-.866) {};
    \node[filled, minimum width = 4pt] (15) at (-.5,-.866) {};
    \node[filled, minimum width = 4pt] (16) at (-.5,.866) {};
    \node[filled, minimum width = 4pt] (17) at (1,0) {};
    \node[filled, minimum width = 4pt] (18) at (-2,0) {};
    \draw[red, thick] (11) -- (12);
    \draw[red, thick] (11) -- (13);
    \draw[red, thick] (11) -- (14);
    \draw[black, thick] (11) -- (15);
    \draw[black, thick] (11) -- (16);
    \draw[black, thick] (11) -- (17);
    \draw[black, thick] (12) -- (18);
    
    \node[filled, red, minimum width = 4pt] (111) at (0+3.5,0) {};
    \node[filled, red, minimum width = 4pt] (112) at (-1+3.5,0) {};
    \node[filled, red, minimum width = 4pt] (113) at (.5+3.5,.866) {};
    \node[filled, red, minimum width = 4pt] (114) at (.5+3.5,-.866) {};
    \node[filled, minimum width = 4pt] (115) at (-1.5+3.5,-.866) {};
    \node[filled, minimum width = 4pt] (116) at (-1.5+3.5,.866) {};
    \node[filled, minimum width = 4pt] (117) at (-1.94+3.5,.342) {};
    \node[filled, minimum width = 4pt] (118) at (-1.94+3.5,-.342) {};
    \draw[red, thick] (111) -- (112);
    \draw[red, thick] (111) -- (113);
    \draw[red, thick] (111) -- (114);
    \draw[black, thick] (112) -- (115);
    \draw[black, thick] (112) -- (116);
    \draw[black, thick] (112) -- (117);
    \draw[black, thick] (112) -- (118);

\end{tikzpicture}
          \begin{tikzpicture}
    \node[filled, red, minimum width = 4pt] (2221) at (0,0) {};
    \node[filled, red, minimum width = 4pt] (2222) at (-1,0) {};
    \node[filled, red, minimum width = 4pt] (2223) at (.5,.866) {};
    \node[filled, red, minimum width = 4pt] (2224) at (.5,-.866) {};
    \node[filled, minimum width = 4pt] (2225) at (-2,0) {};
    \node[filled, minimum width = 4pt] (2226) at (1,0) {};
    \node[filled, minimum width = 4pt] (2227) at (1,1.73) {};
    \node[filled, minimum width = 4pt] (2228) at (1,-1.73) {};
    \draw[red, thick] (2221) -- (2222);
    \draw[red, thick] (2221) -- (2223);
    \draw[red, thick] (2221) -- (2224);
    \draw[black, thick] (2222) -- (2225);
    \draw[black, thick] (2221) -- (2226);
    \draw[black, thick] (2223) -- (2227);
    \draw[black, thick] (2224) -- (2228);
    
    \node[filled, red, minimum width = 4pt] (22221) at (0+3.5,0) {};
    \node[filled, red, minimum width = 4pt] (22222) at (-1+3.5,0) {};
    \node[filled, red, minimum width = 4pt] (22223) at (.5+3.5,.866) {};
    \node[filled, red, minimum width = 4pt] (22224) at (.5+3.5,-.866) {};
    \node[filled, minimum width = 4pt] (22225) at (-1.5+3.5,.866) {};
    \node[filled, minimum width = 4pt] (22226) at (-1.5+3.5,-.866) {};
    \node[filled, minimum width = 4pt] (22227) at (1+3.5,1.73) {};
    \node[filled, minimum width = 4pt] (22228) at (1+3.5,-1.73) {};
    \draw[red, thick] (22221) -- (22222);
    \draw[red, thick] (22221) -- (22223);
    \draw[red, thick] (22221) -- (22224);
    \draw[black, thick] (22222) -- (22225);
    \draw[black, thick] (22222) -- (22226);
    \draw[black, thick] (22223) -- (22227);
    \draw[black, thick] (22224) -- (22228);
\end{tikzpicture}
          \begin{tikzpicture}
    \node[filled, red, minimum width = 4pt] (21) at (0,0) {};
    \node[filled, red, minimum width = 4pt] (22) at (-1,0) {};
    \node[filled, red, minimum width = 4pt] (23) at (.5,.866) {};
    \node[filled, red, minimum width = 4pt] (24) at (.5,-.866) {};
    \node[filled, minimum width = 4pt] (25) at (.5,1.866) {};
    \node[filled, minimum width = 4pt] (26) at (.5,-1.866) {};
    \node[filled, minimum width = 4pt] (27) at (1.5,.866) {};
    \node[filled, minimum width = 4pt] (28) at (1.5,-.866) {};
    \draw[red, thick] (21) -- (22);
    \draw[red, thick] (21) -- (23);
    \draw[red, thick] (21) -- (24);
    \draw[black, thick] (23) -- (25);
    \draw[black, thick] (24) -- (26);
    \draw[black, thick] (23) -- (27);
    \draw[black, thick] (24) -- (28);

    \node[filled, red, minimum width = 4pt] (221) at (0+3.5,0) {};
    \node[filled, red, minimum width = 4pt] (222) at (-1+3.5,0) {};
    \node[filled, red, minimum width = 4pt] (223) at (.5+3.5,.866) {};
    \node[filled, red, minimum width = 4pt] (224) at (.5+3.5,-.866) {};
    \node[filled, minimum width = 4pt] (225) at (-.5+3.5,-.866) {};
    \node[filled, minimum width = 4pt] (226) at (-.5+3.5,.866) {};
    \node[filled, minimum width = 4pt] (227) at (1+3.5,1.73) {};
    \node[filled, minimum width = 4pt] (228) at (1+3.5,-1.73) {};
    \draw[red, thick] (221) -- (222);
    \draw[red, thick] (221) -- (223);
    \draw[red, thick] (221) -- (224);
    \draw[black, thick] (221) -- (225);
    \draw[black, thick] (221) -- (226);
    \draw[black, thick] (223) -- (227);
    \draw[black, thick] (224) -- (228);

    \node[filled, red, minimum width = 4pt] (21) at (0-3.5,0) {};
    \node[filled, red, minimum width = 4pt] (22) at (-1-3.5,0) {};
    \node[filled, red, minimum width = 4pt] (23) at (.5-3.5,.866) {};
    \node[filled, red, minimum width = 4pt] (24) at (.5-3.5,-.866) {};
    \node[filled, minimum width = 4pt] (25) at (-1.5-3.5,-.866) {};
    \node[filled, minimum width = 4pt] (26) at (-1.5-3.5,.866) {};
    \node[filled, minimum width = 4pt] (27) at (1-3.5,0) {};
    \node[filled, minimum width = 4pt] (28) at (1-3.5,1.73) {};
    \draw[red, thick] (21) -- (22);
    \draw[red, thick] (21) -- (23);
    \draw[red, thick] (21) -- (24);
    \draw[black, thick] (22) -- (25);
    \draw[black, thick] (22) -- (26);
    \draw[black, thick] (21) -- (27);
    \draw[black, thick] (23) -- (28);
\end{tikzpicture}
          \begin{tikzpicture}
    \node[filled, red, minimum width = 4pt] (1111) at (0,0) {};
    \node[filled, red, minimum width = 4pt] (1112) at (-1,0) {};
    \node[filled, red, minimum width = 4pt] (1113) at (.5,.866) {};
    \node[filled, red, minimum width = 4pt] (1114) at (.5,-.866) {};
    \node[filled, minimum width = 4pt] (1115) at (-1.5,-.866) {};
    \node[filled, minimum width = 4pt] (1116) at (-1.5,.866) {};
    \node[filled, minimum width = 4pt] (1117) at (-2,0) {};
    \node[filled, minimum width = 4pt] (1118) at (1,0) {};
    \draw[red, thick] (1111) -- (1112);
    \draw[red, thick] (1111) -- (1113);
    \draw[red, thick] (1111) -- (1114);
    \draw[black, thick] (1112) -- (1115);
    \draw[black, thick] (1112) -- (1116);
    \draw[black, thick] (1112) -- (1117);
    \draw[black, thick] (1111) -- (1118);
    
    \node[filled, red, minimum width = 4pt] (11111) at (0+3.5,0) {};
    \node[filled, red, minimum width = 4pt] (11112) at (-1+3.5,0) {};
    \node[filled, red, minimum width = 4pt] (11113) at (.5+3.5,.866) {};
    \node[filled, red, minimum width = 4pt] (11114) at (.5+3.5,-.866) {};
    \node[filled, minimum width = 4pt] (11115) at (-1.5+3.5,-.866) {};
    \node[filled, minimum width = 4pt] (11116) at (-1.5+3.5,.866) {};
    \node[filled, minimum width = 4pt] (11117) at (-2+3.5,0) {};
    \node[filled, minimum width = 4pt] (11118) at (1+3.5,1.73) {};
    \draw[red, thick] (11111) -- (11112);
    \draw[red, thick] (11111) -- (11113);
    \draw[red, thick] (11111) -- (11114);
    \draw[black, thick] (11112) -- (11115);
    \draw[black, thick] (11112) -- (11116);
    \draw[black, thick] (11112) -- (11117);
    \draw[black, thick] (11113) -- (11118);
    
\end{tikzpicture}
        \end{center}
        \caption{Ten trees from $S_3$ via Procedure \ref{proc I}}
        \label{fig:Ten Trees}
     \end{figure}

    Next, the 5 trees in Figure \ref{fig:Five Trees} result from applying Procedure \ref{proc II} to $S_3$ and thus admit a $\mathbb{Z}_2^2$-cordial labeling.
        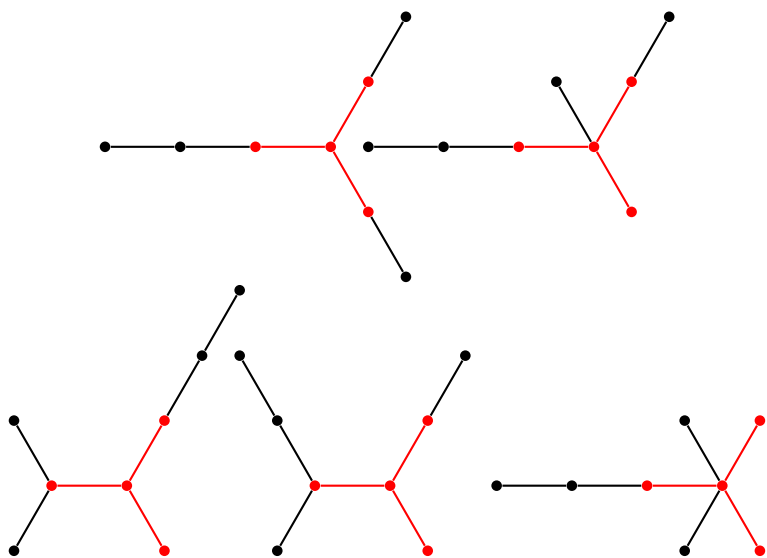
\begin{figure}[H]
        \begin{center}
          \begin{tikzpicture}

    \node[filled, red, minimum width = 4pt] (1) at (0-3.5,0) {};
    \node[filled, red, minimum width = 4pt] (2) at (-1-3.5,0) {};
    \node[filled, red, minimum width = 4pt] (3) at (.5-3.5,.866) {};
    \node[filled, red, minimum width = 4pt] (4) at (.5-3.5,-.866) {};
    \node[filled, minimum width = 4pt] (5) at (-7.5+2,0) {};
    \node[filled, minimum width = 4pt] (6) at (-8.5+2,0) {};
    \node[filled, minimum width = 4pt] (7) at (-4.5+2,1.73) {};
    \node[filled, minimum width = 4pt] (8) at (-4.5+2,-1.73) {};
    \draw[red, thick] (1) -- (2);
    \draw[red, thick] (1) -- (3);
    \draw[red, thick] (1) -- (4);
    \draw[black, thick] (2) -- (5);
    \draw[black, thick] (5) -- (6);
    \draw[black, thick] (3) -- (7);
    \draw[black, thick] (4) -- (8);

    \node[filled, red, minimum width = 4pt] (11) at (0,0) {};
    \node[filled, red, minimum width = 4pt] (12) at (-1,0) {};
    \node[filled, red, minimum width = 4pt] (13) at (.5,.866) {};
    \node[filled, red, minimum width = 4pt] (14) at (.5,-.866) {};
    \node[filled, minimum width = 4pt] (15) at (-2,0) {};
    \node[filled, minimum width = 4pt] (16) at (-3,0) {};
    \node[filled, minimum width = 4pt] (17) at (1,1.73) {};
    \node[filled, minimum width = 4pt] (18) at (-.5,.866) {};
    \draw[red, thick] (11) -- (12);
    \draw[red, thick] (11) -- (13);
    \draw[red, thick] (11) -- (14);
    \draw[black, thick] (12) -- (15);
    \draw[black, thick] (15) -- (16);
    \draw[black, thick] (13) -- (17);
    \draw[black, thick] (11) -- (18);

\end{tikzpicture}
          \begin{tikzpicture}

    \node[filled, red, minimum width = 4pt] (211) at (0,0) {};
    \node[filled, red, minimum width = 4pt] (212) at (-1,0) {};
    \node[filled, red, minimum width = 4pt] (213) at (.5,.866) {};
    \node[filled, red, minimum width = 4pt] (214) at (.5,-.866) {};
    \node[filled, minimum width = 4pt] (215) at (-1.5,-.866) {};
    \node[filled, minimum width = 4pt] (216) at (-1.5,.866) {};
    \node[filled, minimum width = 4pt] (217) at (1,1.73) {};
    \node[filled, minimum width = 4pt] (218) at (1.5,2.60) {};
    \draw[red, thick] (211) -- (212);
    \draw[red, thick] (211) -- (213);
    \draw[red, thick] (211) -- (214);
    \draw[black, thick] (212) -- (215);
    \draw[black, thick] (212) -- (216);
    \draw[black, thick] (213) -- (217);
    \draw[black, thick] (217) -- (218);
    
    \node[filled, red, minimum width = 4pt] (2111) at (3+.5,0) {};
    \node[filled, red, minimum width = 4pt] (2112) at (2+.5,0) {};
    \node[filled, red, minimum width = 4pt] (2113) at (3.5+.5,.866) {};
    \node[filled, red, minimum width = 4pt] (2114) at (3.5+.5,-.866) {};
    \node[filled, minimum width = 4pt] (2115) at (1.5+.5,-.866) {};
    \node[filled, minimum width = 4pt] (2116) at (1.5+.5,.866) {};
    \node[filled, minimum width = 4pt] (2117) at (4+.5,1.73) {};
    \node[filled, minimum width = 4pt] (2118) at (1+.5,1.73) {};
    \draw[red, thick] (2111) -- (2112);
    \draw[red, thick] (2111) -- (2113);
    \draw[red, thick] (2111) -- (2114);
    \draw[black, thick] (2112) -- (2115);
    \draw[black, thick] (2112) -- (2116);
    \draw[black, thick] (2113) -- (2117);
    \draw[black, thick] (2116) -- (2118);

\end{tikzpicture}
          \begin{tikzpicture}
    \node[filled, red, minimum width = 4pt] (111) at (0,0) {};
    \node[filled, red, minimum width = 4pt] (112) at (-1,0) {};
    \node[filled, red, minimum width = 4pt] (113) at (.5,.866) {};
    \node[filled, red, minimum width = 4pt] (114) at (.5,-.866) {};
    \node[filled, minimum width = 4pt] (115) at (-2,0) {};
    \node[filled, minimum width = 4pt] (116) at (-3,0) {};
    \node[filled, minimum width = 4pt] (117) at (-.5,.866) {};
    \node[filled, minimum width = 4pt] (118) at (-.5,-.866) {};
    \draw[red, thick] (111) -- (112);
    \draw[red, thick] (111) -- (113);
    \draw[red, thick] (111) -- (114);
    \draw[black, thick] (112) -- (115);
    \draw[black, thick] (115) -- (116);
    \draw[black, thick] (111) -- (117);
    \draw[black, thick] (111) -- (118);
\end{tikzpicture}
        \end{center}
        \caption{Five trees from $S_3$ via Procedure \ref{proc II}}
        \label{fig:Five Trees}
     \end{figure}

    The last 8 trees must be verified by hand, which is done in Figure~\ref{trees verified by hand}.
      \begin{figure}[H]
        \begin{tikzpicture}[yscale=0.75, xscale=0.75]

    \node[filled, minimum width = 4pt] (1) at (-4,0)[label=0]  {};
    \node[filled, minimum width = 4pt] (2) at (-3,0)[label=a]  {};
    \node[filled, minimum width = 4pt] (3) at (-2,0)[label=b]  {};
    \node[filled, minimum width = 4pt] (4) at (-1,0)[label=c]  {};
    \node[filled, minimum width = 4pt] (5) at (0,0)[label=c]  {};
    \node[filled, minimum width = 4pt] (6) at (1,0)[label=a]  {};
	\node[filled, minimum width = 4pt] (7) at (2,0)[label=b]  {};
	\node[filled, minimum width = 4pt] (8) at (3,0)[label=0]  {};
	\node[filled, minimum width = 4pt] (9) at (1,-1.7)[color=white]  {};

    \draw[black, thick] (1) -- (2) node [midway,below] {a};
    \draw[black, thick] (2) -- (3) node [midway,below] {c};
    \draw[black, thick] (3) -- (4) node [midway,below] {a};
    \draw[black, thick] (4) -- (5) node [midway,below] {0};
    \draw[black, thick] (5) -- (6) node [midway,below] {b};
    \draw[black, thick] (6) -- (7) node [midway,below] {c};
    \draw[black, thick] (7) -- (8) node [midway,below] {b};

\end{tikzpicture}	
\begin{tikzpicture}[yscale=0.75, xscale=0.75]
		
		\node[filled, minimum width = 4pt] (1) at (-4,0)[label=0]  {};
		\node[filled, minimum width = 4pt] (2) at (-3,0)[label=a]  {};
		\node[filled, minimum width = 4pt] (3) at (-2,0)[label=b]  {};
		\node[filled, minimum width = 4pt] (4) at (-1,0)[label=b]  {};
		\node[filled, minimum width = 4pt] (5) at (0,0)[label=c]  {};
		\node[filled, minimum width = 4pt] (6) at (1,0)[label=c]  {};
		\node[filled, minimum width = 4pt] (7) at (2,0)[label=a]  {};
		\node[filled, minimum width = 4pt] (8) at (1,-1)[label=below:0]  {};

		\draw[black, thick] (1) -- (2) node [midway,below] {a};
		\draw[black, thick] (2) -- (3) node [midway,below] {c};
		\draw[black, thick] (3) -- (4) node [midway,below] {0};
		\draw[black, thick] (4) -- (5) node [midway,below] {a};
		\draw[black, thick] (5) -- (6) node [midway,below] {0};
		\draw[black, thick] (6) -- (7) node [midway,below] {b};
		\draw[black, thick] (6) -- (8) node [near end, right] {c};
		
	\end{tikzpicture}
	\begin{tikzpicture}[yscale=0.75, xscale=0.75]
		
		\node[filled, minimum width = 4pt] (1) at (-4,0)[label=a]  {};
		\node[filled, minimum width = 4pt] (2) at (-3,0)[label=b]  {};
		\node[filled, minimum width = 4pt] (3) at (-2,0)[label=b]  {};
		\node[filled, minimum width = 4pt] (4) at (-1,0)[label=c]  {};
		\node[filled, minimum width = 4pt] (5) at (0,0)[label=c]  {};
		\node[filled, minimum width = 4pt] (6) at (1,0)[label=a]  {};
		\node[filled, minimum width = 4pt] (7) at (-3,-1)[label=below:0]  {};
		\node[filled, minimum width = 4pt] (8) at (0,-1)[label=below:0]  {};

		\draw[black, thick] (1) -- (2) node [midway,below] {c};
		\draw[black, thick] (2) -- (3) node [midway,below] {0};
		\draw[black, thick] (3) -- (4) node [midway,below] {a};
		\draw[black, thick] (4) -- (5) node [midway,below] {0};
		\draw[black, thick] (5) -- (6) node [midway,below] {b};
		\draw[black, thick] (2) -- (7) node [near end, right] {b};
		\draw[black, thick] (5) -- (8) node [near end, right] {c};
	\end{tikzpicture}
        	\begin{tikzpicture}[yscale=0.75, xscale=0.75]
		
		\node[filled, minimum width = 4pt] (1) at (-4,0)[label=a]  {};
		\node[filled, minimum width = 4pt] (2) at (-3,0)[label=b]  {};
		\node[filled, minimum width = 4pt] (3) at (-2,0)[label=c]  {};
		\node[filled, minimum width = 4pt] (4) at (-1,0)[label=a]  {};
		\node[filled, minimum width = 4pt] (5) at (0,0)[label=above left:0]  {};
		\node[filled, minimum width = 4pt] (6) at (1,0)[label=0]  {};
		\node[filled, minimum width = 4pt] (7) at (0,1)[label=above:b]  {};
		\node[filled, minimum width = 4pt] (8) at (0,-1)[label=below:c]  {};

		\draw[black, thick] (1) -- (2) node [midway,below] {c};
		\draw[black, thick] (2) -- (3) node [midway,below] {a};
		\draw[black, thick] (3) -- (4) node [midway,below] {b};
		\draw[black, thick] (4) -- (5) node [midway,below] {a};
		\draw[black, thick] (5) -- (6) node [midway,below] {0};
		\draw[black, thick] (5) -- (7) node [near end, right] {b};
		\draw[black, thick] (5) -- (8) node [near end, right] {c};
		
	\end{tikzpicture}
	\begin{tikzpicture}[yscale=0.75, xscale=0.75]
		
		\node[filled, minimum width = 4pt] (1) at (-4,0)[label=0]  {};
		\node[filled, minimum width = 4pt] (2) at (-3,0)[label=a]  {};
		\node[filled, minimum width = 4pt] (3) at (-2,0)[label=b]  {};
		\node[filled, minimum width = 4pt] (4) at (-1,0)[label=b]  {};
		\node[filled, minimum width = 4pt] (5) at (0,0)[label=c]  {};
		\node[filled, minimum width = 4pt] (6) at (1,0)[label=c]  {};
		\node[filled, minimum width = 4pt] (7) at (2,0)[label=a]  {};
		\node[filled, minimum width = 4pt] (8) at (0,-1)[label=below:0]  {};

		\draw[black, thick] (1) -- (2) node [midway,below] {a};
		\draw[black, thick] (2) -- (3) node [midway,below] {c};
		\draw[black, thick] (3) -- (4) node [midway,below] {0};
		\draw[black, thick] (4) -- (5) node [midway,below] {a};
		\draw[black, thick] (5) -- (6) node [midway,below] {0};
		\draw[black, thick] (6) -- (7) node [midway,below] {b};
		\draw[black, thick] (5) -- (8) node [near end, right] {c};
		
	\end{tikzpicture}
	\begin{tikzpicture}[yscale=0.75, xscale=0.75]
		
		\node[filled, minimum width = 4pt] (1) at (-4,0)[label=a]  {};
		\node[filled, minimum width = 4pt] (2) at (-3,0)[label=b]  {};
		\node[filled, minimum width = 4pt] (3) at (-2,0)[label=b]  {};
		\node[filled, minimum width = 4pt] (4) at (-1,0)[label=c]  {};
		\node[filled, minimum width = 4pt] (5) at (0,0)[label=c]  {};
		\node[filled, minimum width = 4pt] (6) at (1,0)[label=a]  {};
		\node[filled, minimum width = 4pt] (7) at (-3,-1)[label=below:0]  {};
		\node[filled, minimum width = 4pt] (8) at (-1,-1)[label=below:0]  {};

		\draw[black, thick] (1) -- (2) node [midway,below] {c};
		\draw[black, thick] (2) -- (3) node [midway,below] {0};
		\draw[black, thick] (3) -- (4) node [midway,below] {a};
		\draw[black, thick] (4) -- (5) node [midway,below] {0};
		\draw[black, thick] (5) -- (6) node [midway,below] {b};
		\draw[black, thick] (2) -- (7) node [near end, right] {b};
		\draw[black, thick] (4) -- (8) node [near end, right] {c};
		
	\end{tikzpicture}
        	\begin{tikzpicture}[yscale=0.75, xscale=0.75]
		
		\node[filled, minimum width = 4pt] (1) at (-4,0)[label=a]  {};
		\node[filled, minimum width = 4pt] (2) at (-3,0)[label=b]  {};
		\node[filled, minimum width = 4pt] (3) at (-2,0)[label=b]  {};
		\node[filled, minimum width = 4pt] (4) at (-1,0)[label=c]  {};
		\node[filled, minimum width = 4pt] (5) at (0,0)[label=c]  {};
		\node[filled, minimum width = 4pt] (6) at (1,0)[label=a]  {};
		\node[filled, minimum width = 4pt] (7) at (2,0)[label=0]  {};
		\node[filled, minimum width = 4pt] (8) at (-1,-1)[label=below:0]  {};

		\draw[black, thick] (1) -- (2) node [midway,below] {c};
		\draw[black, thick] (2) -- (3) node [midway,below] {0};
		\draw[black, thick] (3) -- (4) node [midway,below] {a};
		\draw[black, thick] (4) -- (5) node [midway,below] {0};
		\draw[black, thick] (5) -- (6) node [midway,below] {b};
		\draw[black, thick] (6) -- (7) node [midway,below] {a};
		\draw[black, thick] (4) -- (8) node [near end, right] {c};
		
	\end{tikzpicture}
	\begin{tikzpicture}[yscale=0.75, xscale=0.75]
		
		\node[filled, minimum width = 4pt] (1) at (-2,1)[label=above:a]  {};
		\node[filled, minimum width = 4pt] (2) at (-2,-1)[label=below:0]  {};
		\node[filled, minimum width = 4pt] (3) at (-2,0)[label=left:b]  {};
		\node[filled, minimum width = 4pt] (4) at (-1,0)[label=b]  {};
		\node[filled, minimum width = 4pt] (5) at (0,0)[label=above left:c]  {};
		\node[filled, minimum width = 4pt] (6) at (1,0)[label=c]  {};
		\node[filled, minimum width = 4pt] (7) at (0,1)[label=above:a]  {};
		\node[filled, minimum width = 4pt] (8) at (0,-1)[label=below:0]  {};

		\draw[black, thick] (1) -- (3) node [midway,right] {c};
		\draw[black, thick] (2) -- (3) node [midway,right] {b};
		\draw[black, thick] (3) -- (4) node [midway,below] {0};
		\draw[black, thick] (4) -- (5) node [midway,below] {a};
		\draw[black, thick] (5) -- (6) node [midway,below] {0};
		\draw[black, thick] (5) -- (7) node [midway, right] {b};
		\draw[black, thick] (5) -- (8) node [midway, right] {c};
		
	\end{tikzpicture}
      \caption{Eight remaining trees}
      \label{trees verified by hand}
     \end{figure}
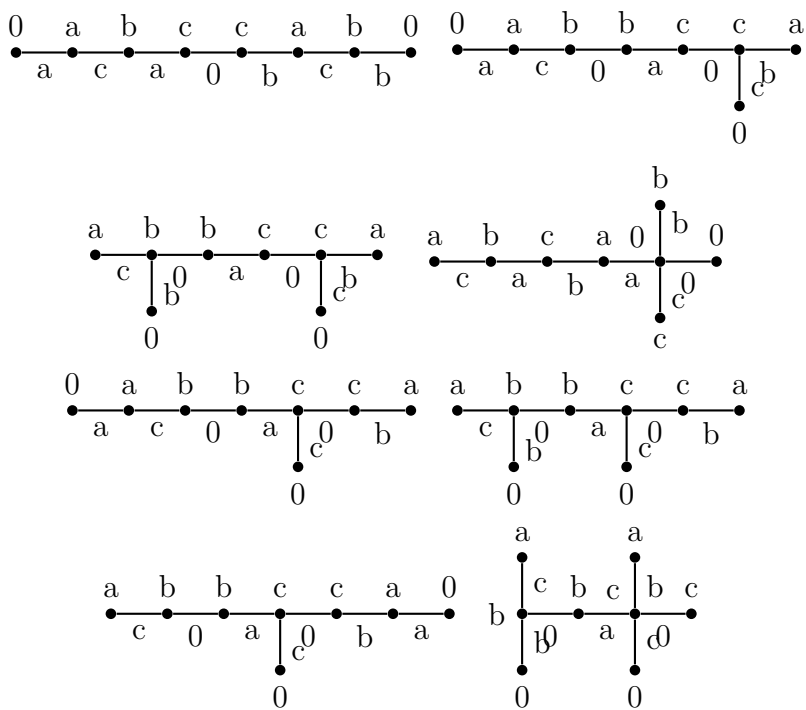

    Now, suppose that $k\geq 2$. Let $T^*$ be any tree with $|V(T^*)| = 4k+4$. We have three cases:
    \begin{enumerate}
        \item $T^*$ has four or more leaves, in which case $T^*$ can be obtained by applying Procedure \ref{proc I} to to a $\mathbb{Z}_2^2$-cordial tree $T$ with $|V(T)|=4k$.
        \item $T^*$ has exactly three leaves, in which case $T^*$ can be obtained by applying Procedure \ref{proc II} to a $\mathbb{Z}_2^2$-cordial tree $T$ with $|V(T)| = 4k$.
        \item $T^*$ has exactly 2 leaves, in which case $T^*$ is a path meaning $T^*$ is $\mathbb{Z}_2^2$-cordial \cite[Theorem 3.4]{pechenikWise}.\qedhere
    \end{enumerate}
\end{proof}

\section{$\mathbb{Z}_2^2$-Cordiality for Trees}

We come now to our main result on tree cordiality.

\begin{theorem} \label{theorem: main result on trees}
    All trees, except $P_4$ and $P_5$, are $\mathbb{Z}_2^2$-cordial.
\end{theorem}
\begin{proof}
    Combining Theorem \ref{extension of 4k vertices to 4k+4 verts} and Lemma \ref{extending 4k to 4k+3}, we are left to consider trees $T^*$ which are an extension of $P_4$ such that $|V(T^*)|\leq 7.$

    If $T^*$ has three or more leaves, then it is easy to see that $T^*$ is an extension of the $\mathbb{Z}_2^2$-cordial graph $S_3$ and is thus $\mathbb{Z}_2^2$-cordial by Lemma~\ref{extending 4k to 4k+3}.
    If, on the other hand, $T^*$ has exactly two leaves, then $T^*$ is a path and, thus, is $\mathbb{Z}_2^2$-cordial if and only if $T^*$ is \emph{not} $P_4$ or $P_5$ as shown in \cite[Theorem 3.4]{pechenikWise}.
\end{proof}

We end this section with a definition and conjecture.

\begin{definition}
    A family of $\mathcal{G}$ graphs is called \emph{weakly $A$-cordial} if
    all but a finite number of elements of $\mathcal{G}$ are $A$-cordial.
\end{definition}

For example, Theorem \ref{theorem: main result on trees} shows that the family of trees is weakly $\mathbb{Z}_2^2$-cordial.

\begin{conjecture} \label{conjecture: weakly cordial}
    For any abelian group $A$, the set of trees is weakly $A$-cordial.
\end{conjecture}

\section{General Obstruction to Friendship Cordiality} \label{section: obstruction to friendship cordiality}

We now turn our attention to friendship graphs.
Below is a general obstruction to friendship cyclic cordiality.

\begin{theorem} \label{theorem: obstruction}
    Let $m, n \in \mathbb{N}$ with $m \geq 1$. If
    \begin{enumerate}
        \item $m, n$ are even
        \item $m \mid 3n$
        \item $4 \nmid n$,
    \end{enumerate}
    then $F_n$ is not $\mathbb{Z}_m$-cordial.
\end{theorem}

\begin{proof}
    As $F_n$ has $3n$ edges and $m \mid 3n$, a cordial edge label frequency would be uniformly $\frac{3n}{m}$. As $m$ is even, the sum of all the elements of $\mathbb{Z}_m$ is $\frac{m}{2}$ so the sum of all the edges of a $\mathbb{Z}_m$-cordial labeling would be $\frac{3n}{m} \frac{m}{2} = \frac{3n}{2}$. If $4 \nmid n$, this is odd (and well-defined in $\mathbb{Z}_m$ as $m$ is even). On the other hand, the sum of all the edge labels is twice the sum of the circumferential vertex labels plus $2n$ times the central vertex label, which is even. Contradiction!
\end{proof}

\begin{conjecture} \label{conjecture for cordiality of friendship}
    Let $m, n \in \mathbb{N}$ with $m \geq 1$.
    Then $F_n$ is $\mathbb{Z}_m$-cordial unless
    \begin{enumerate}
        \item $m, n$ are even
        \item $m \mid 3n$
        \item $4 \nmid n$.
    \end{enumerate}
\end{conjecture}

Note that the above condition on $m, n$ is equivalent to requiring that
\[ 2 \mid n, \,\, 4 \nmid n, \,\, m = \frac{3n}{d} \]
for $d$ a positive, \emph{odd} divisor of $3n$.

\section{Friendship Reductions}

We begin with an important labeling convention.
By Lemma \ref{lemma: other cordial labelings}, just using label shifting, we may always assume that a $\mathbb{Z}_m$-cordial labeling of $F_n$ labels the central vertex with $0$. We will always do so.

\begin{definition}\label{def: adding friendship}
    Let $n_1, n_2 \in \mathbb{N}$. If $\ell_i:V_i \rightarrow A$, $i = 1, 2$, are labelings of $F_{n_i}$ with the central vertex labeled by $0$ by both $\ell_i$, write
    \[ F_{n_1} + F_{n_2} \]
    for the labeling of $F_{n_1 + n_2}$ using $\ell_1$ on the first $n_1$ triangles and $\ell_2$ on the next $n_2$ triangles.
\end{definition}

For example, Figure \ref{fig:example of adding cordial friendships} illustrates the sum $F_3 = F_1 + F_2$.
\begin{figure}[H]
    \begin{subfigure}{.4\textwidth}
      \centering
      \begin{tikzpicture}
   
    \node[filled, minimum width = 4pt, label = {270:0}] (0) at (0,0) {};
    \node[filled, minimum width = 4pt, label={60:1}] (1) at (1*60: 2cm) {};
    \node[filled, minimum width = 4pt, label={120:2}] (2) at (2*60: 2cm) {};

    \draw[black, thick] (0) -- (1) node [midway, fill=white] {1};
    \draw[black, thick] (0) -- (2) node [midway, fill=white] {2};
    \draw[black, thick] (1) -- (2) node [midway, fill=white] {3};

\end{tikzpicture}
      \caption{$F_1$}
    \end{subfigure}
    \hfill
    \begin{subfigure}{.4\textwidth}
      \centering
      \begin{tikzpicture}
    \node[filled, minimum width = 4pt, label = {90:0}] (0) at (0,0) {};
    \node[filled, minimum width = 4pt, label={180:9}] (9) at (3*60: 2cm) {};
    \node[filled, minimum width = 4pt, label={240:4}] (4) at (4*60: 2cm) {};
    \node[filled, minimum width = 4pt, label={300: 5}] (5) at (5*60: 2cm) {};
    \node[filled, minimum width = 4pt, label={360: 6}] (6) at (6*60: 2cm) {};
    
    \draw[black, thick] (0) -- (9) node [midway, fill=white] {9};
    \draw[black, thick] (0) -- (4) node [midway, fill=white] {4};
    \draw[black, thick] (9) -- (4) node [midway, fill=white] {13};
    \draw[black, thick] (0) -- (5) node [midway, fill=white] {5};
    \draw[black, thick] (0) -- (6) node [midway, fill=white] {6};
    \draw[black, thick] (5) -- (6) node [midway, fill=white] {11};
\end{tikzpicture}
      \caption{$F_2$}
    \end{subfigure}
    \hfill
    \begin{subfigure}{.4\textwidth}
      \centering
      \begin{tikzpicture}
    \node[filled, minimum width = 4pt, label = {270:0}] (0) at (0,0) {};
    \node[filled, minimum width = 4pt, label={60:1}] (1) at (1*60: 2cm) {};
    \node[filled, minimum width = 4pt, label={120:2}] (2) at (2*60: 2cm) {};
    \node[filled, minimum width = 4pt, label={180:9}] (9) at (3*60: 2cm) {};
    \node[filled, minimum width = 4pt, label={240:4}] (4) at (4*60: 2cm) {};
    \node[filled, minimum width = 4pt, label={300: 5}] (5) at (5*60: 2cm) {};
    \node[filled, minimum width = 4pt, label={360: 6}] (6) at (6*60: 2cm) {};
    
    \draw[black, thick] (0) -- (1) node [midway, fill=white] {1};
    \draw[black, thick] (0) -- (2) node [midway, fill=white] {2};
    \draw[black, thick] (1) -- (2) node [midway, fill=white] {3};
    \draw[black, thick] (0) -- (9) node [midway, fill=white] {9};
    \draw[black, thick] (0) -- (4) node [midway, fill=white] {4};
    \draw[black, thick] (9) -- (4) node [midway, fill=white] {13};
    \draw[black, thick] (0) -- (5) node [midway, fill=white] {5};
    \draw[black, thick] (0) -- (6) node [midway, fill=white] {6};
    \draw[black, thick] (5) -- (6) node [midway, fill=white] {11};
\end{tikzpicture}
      \caption{$F_3 = F_1 + F_2$}
    \end{subfigure}
    \hfill
    \caption{Adding friendship graphs}
    \label{fig:example of adding cordial friendships}
\end{figure}

The following lemma follows immediately from a distribution count.

\begin{lemma} \label{lemma: adding friendship}
    Let $m, n_1, n_2 \in \mathbb{N}$ with $m \geq 1$. Let $\ell_i:V_i \rightarrow \mathbb{Z}_m$, $i = 1, 2$, be labelings of $F_{n_i}$ with the central vertex labeled by $0$ for both $\ell_i$.
    Suppose both $F_{n_i}$ have $\mathbb{Z}_m$-cordial labelings. Assume also that $F_{n_2}$ has a distribution of
    \[ f_{V_2}(0) = r + 1, \,\, f_{V_2}(i) = r \,\,
       \text{for} \,\, i \not= 0, \,\, \text{and} \,\,
       f_{E_2}(j) = s \,\, \text{for all} \,\, j, \]
    for some $r, s$. Then
    \[ F_{n_1} + F_{n_2} \]
    is also $\mathbb{Z}_m$-cordial.
\end{lemma}

Now Theorem \ref{theorem: free cordial after} shows that when
\begin{equation}
    m > 2n(3n + 1),
\end{equation}
then $F_n$ is $\mathbb{Z}_m$-cordial. This bound is not sharp. Further progress requires some reductions.

The next result reduces verifying cordiality in general to checking certain small cases. Note that by dint of part $(c)$ below, part $(b)$ is genuinely useful only when $4 \mid m$.

\begin{theorem} \label{theorem: Fm or F2m reduction}
    Let $m \in \mathbb{N}$ with $m \geq 1$.
    \begin{enumerate}
        \item[$(a)$] If $m$ is odd, then $F_m$ is $\mathbb{Z}_m$-cordial. Moreover, $F_n$ is $\mathbb{Z}_m$-cordial for all $n \in \mathbb{N}$ if and only if $F_1, F_2, \ldots, F_{m-1}$ are $\mathbb{Z}_m$-cordial.
        \item[$(b)$] If $m$ is even, then $F_{2m}$ is $\mathbb{Z}_m$-cordial. Moreover, $F_n$ is $\mathbb{Z}_m$-cordial for all $n \in \mathbb{N}$ if and only if $F_1, F_2, \ldots, F_{2m-1}$ are $\mathbb{Z}_m$-cordial.
        \item[$(c)$] If $m$ is even, but $4 \nmid m$, then $F_m$ is not $\mathbb{Z}_m$-cordial.
    \end{enumerate}
\end{theorem}

\begin{proof}
    For part (a), label the central vertex of $F_m$ with $0$. Then add $m$ triangles whose circumferential vertex labels are the pairs
    \[ (0, 1), (2, 3), (4, 5) \ldots, (m-1, 0),
       (1, 2), (3, 4), (5, 6), \ldots, (m-2, m-1). \]
    This gives $f_V(0) = 3$ and $f_V(i) = 2$ for $i \not= 0$ and so is vertex-cordial.
    The above numbers also give the radial edge sums. The circumferential edge sums are $1 + 4i$ for $0 \leq i \leq m - 1$. As $4$ is a unit of $\mathbb{Z}_m$, these sums cycle through all of $\mathbb{Z}_m$. Therefore, $f_E(i) = 3$ for $0 \leq i \leq m - 1$ so that the labeling is edge-cordial as well.

    This gives a $\mathbb{Z}_m$-cordial construction of $F_m$. The general result follows from Lemma \ref{lemma: adding friendship}.

    For part (b), label the central vertex of $F_{2m}$ with $0$. Then add $2m$ triangles whose circumferential vertex labels are the pairs
    \[ (0, 1), (1, 2), (2, 3), \ldots, (m-1, 0),
       (0, 0), (1, 1), (2, 2), \ldots, (m-1, m-1). \]
    This gives $f_V(0) = 5$ and $f_V(i) = 4$ for $i \not= 0$ and so is vertex-cordial.
    The above numbers also give the radial edge sums. The circumferential edge sums are $1 + 2i$ and $2i$ for $0 \leq i \leq m - 1$. As $m$ is even, these sums cycle twice through the odds and evens, respectively, in $\mathbb{Z}_m$. Therefore, $f_E(i) = 6$ for $0 \leq i \leq m - 1$ so that the labeling is edge-cordial as well.
    This gives a $\mathbb{Z}_m$-cordial construction of $F_{2m}$. The general result follows from Lemma \ref{lemma: adding friendship}.

    For part (c), apply Theorem \ref{theorem: obstruction}.
\end{proof}

\section{MAD Pairs and Construction to $F_{\frac{m}{3}}$}

\begin{definition}
    Let $A$ be a finite abelian group and let $\mu = \lfloor\frac{|A|}{3}\rfloor$. We say that $A$ contains $n$ \emph{additively disjoint pairs} if there exists a family $\{(a_i, b_i)\}$ of $n$ pairs, $(a_i, b_i) \in A^2$, $1 \leq i \leq n$, so that
    \[ a_1, b_1, \ldots, a_n, b_n, \, a_1 + b_1, \ldots, a_n + b_n \]
    are all distinct. We say that $A$ is \emph{maximally additively disjoint (MAD)} if it contains $\mu$ additively disjoint pairs. More generally, we will refer to any set of additively disjoint pairs $\{(a_i, b_i)\}$ of maximum size as a family of \emph{MAD pairs}.  If $A$ is MAD, then there will exist $\mu$ MAD pairs.
\end{definition}

Here is our main result on MAD pairs.

\begin{theorem} \label{theorem: MAD}
    The group $\mathbb{Z}_m$ is MAD if and only if $m \not\equiv 6\mod 12$.  If $m \equiv 6 \mod 12$, then $\mathbb{Z}_m$ contains $\mu-1$ MAD pairs.
\end{theorem}

\begin{proof}

We prove this via an algorithm for constructing MAD pairs.  Before giving the algorithm, we choose the set
\begin{equation}
\label{S}
S = \left\{-\floor*{\frac{m-1}{2}}, -\floor*{\frac{m-1}{2}} + 1, \ldots, \floor*{\frac{m}{2}} - 1, \floor*{\frac{m}{2}} \right\}
\end{equation}
of representatives for elements of $\mathbb{Z}_m$.  We regard $S$ itself as a subset of $\mathbb{Z}$, which inherits its ordering and addition from $\mathbb{Z}$. We define
\[
\alpha = \begin{cases}
    \text{nearest integer to $\frac{m}{4}$}, & m \equiv \text{$0$, $1$, $2$, $3$, $4$, or $5$} \mod 12,\\
    \floor*{\frac{m}{4}},& m \equiv 7 \mod 12,\\
    \text{nearest half-integer to $\frac{m}{4}$,}& m \equiv \text{$8$, $9$, $10$, or $11$} \mod 12,
\end{cases}
\]
where we round up in case of a tie.  This $\alpha$ acts somewhat as an axis of symmetry for our algorithm below.  The idea is that we can choose ``half'' of the MAD pairs by starting at $\alpha$ and working our way outward, staying within the positive elements of $S$; then the other ``half'' of the MAD pairs are simply the negatives of the first half.  The algorithm is given below for $\mathbb{Z}_m$ with $m \not\equiv 6 \mod 12$. Example \ref{example:MAD31} and Figure \ref{fig:MAD31} illustrate the algorithm for $m = 31$.

\begin{enumerate}
\setcounter{enumi}{-1}
\item In the special case $m \equiv \text{$0$ or $1$ } \mod 12$, begin by choosing the MAD pair $(a_0, b_0) = (-\floor*{\frac{m}{6}}, \floor*{\frac{m}{6}})$.

\item Let $b_1 \in S$ be the greatest element such that $b_1 < \alpha$, and let $c_1 \in S$ be the least element such that $c_1 > \alpha$.  Then choose the MAD pair $(a_1, b_1)$, where $a_1 = c_1 - b_1$.  Repeat to obtain MAD pairs $(a_i, b_i)$, where $b_i$ decreases by 1 and $c_i$ increases by 1 with each step.  This process must terminate just before the $j$th step, as soon as some $b_j$ has already been used as one of the $a_i$.

\item Now choose the MAD pair $(a_j, b_j)$, where $b_j = c_{j-1} + 1$ and $a_j < \alpha$ is the greatest element that has not yet been used in a MAD pair.  Repeat to obtain MAD pairs $(a_i, b_i)$, where $b_i$ increases by 1 with each step.  This process must terminate once we have used all positive elements less than $\alpha$.  (This always occurs before some $b_k$ is the sum of a previous pair.)

\item For each MAD pair $(a_i, b_i)$ already chosen ($i \geq 1$), also choose its negative pair $( -a_i, -b_i)$.  (For $m \equiv \text{$2$ or $8$ } \mod 12$, this will not be possible for the pair with the greatest sum, since the negative of this sum is not an element of $S$.)

\item[(3*)] In the special case $m \equiv \text{$0$, $1$, or $7$ } \mod 12$, choose the additional pair $(-\alpha, 2\alpha)$.

\item There is now one final MAD pair available, consisting of an element and its negative (except in the case $m \equiv \text{$0$ or $1$ } \mod 12$, where we already chose this as the $0$th pair).
\end{enumerate}

The fact that the algorithm truly produces a family of MAD pairs, i.e., that the $a_i$, $b_i$, and $a_i+b_i$ are all distinct, and that there are $\mu$ such pairs,
is a straightforward calculation. As it involves numerous cases, the explicit calculations are relegated to the Appendix, Section~\ref{appendix}.

It remains to show that $\mathbb{Z}_m$ is not MAD when $m \equiv 6 \mod 12$.  Suppose, to the contrary, that $m = 12k + 6$ and  $\mathbb{Z}_m$ were MAD. Then $\mu = 4k+2$ and there would exist a $\mathbb{Z}_m$-cordial labeling of $F_{\mu}$, where the central vertex is labeled $0$ and the rest of the circumferential vertex pairs are labeled $(a_i, b_i)$ for $1 \leq i \leq \mu$.  However, this contradicts Theorem \ref{theorem: obstruction}, since (1) both $m$ and $\mu$ are even, (2) we have $m \mid 3\mu\ (= m)$, and (3) we have $4 \nmid \mu$. Nevertheless, we can at least find $\mu - 1$ MAD pairs by following the algorithm for the $m \equiv 7 \mod 12$ case and omitting the pair $(-\frac m2, \frac m2)$, see  \eqref{MAD7} in the Appendix, Section \ref{appendix}.
\end{proof}

\begin{example}
\label{example:MAD31}
    Let $m = 31 \equiv 7 \mod 12$.  Then $S = \{-15, \ldots, 15\}$, and $\alpha = 7$. Each step of the algorithm is given below. The results are pictured in
    Figure \ref{fig:MAD31}. Representatives from the set $S$ are placed on the circle. Each MAD pair is shown as a pair of chords whose common vertex is their sum.

    \begin{figure}[H]
    \centering
    \begin{tikzpicture}[scale=4]
\draw(0,0) circle (1);
\pgfmathsetmacro\n{31}
\foreach \i/\k in {0/0,1/$-1$,2/$-2$,3/$-3$,4/$-4$,5/$-5$,6/$-6$,7/$-7$,8/$-8$,9/$-9$,10/$-10$,11/$-11$,12/$-12$,13/$-13$,14/$-14$,15/$-15$,16/15,17/14,18/13,19/12,20/11,21/10,22/9,23/8,24/7,25/6,26/5,27/4,28/3,29/2,30/1} {
    \pgfmathsetmacro\r{\i*(360/\n)+90}
    \fill (\r:1) circle (1pt) coordinate (n-\i);
    \node at (\r:1.2) {\k};
}
\draw[blue] (n-0) -- (n-15);
\draw[blue] (n-0) -- (n-16);
\draw[green] (n-1) -- (n-12);
\draw[green] (n-12) -- (n-11);
\draw[green] (n-10) -- (n-13);
\draw[green] (n-13) -- (n-3);
\draw[green] (n-30) -- (n-19);
\draw[green] (n-19) -- (n-20);
\draw[green] (n-21) -- (n-18);
\draw[green] (n-18) -- (n-28);
\draw[red] (n-2) -- (n-8);
\draw[red] (n-6) -- (n-8);
\draw[red] (n-4) -- (n-9);
\draw[red] (n-9) -- (n-5);
\draw[red] (n-29) -- (n-23);
\draw[red] (n-25) -- (n-23);
\draw[red] (n-27) -- (n-22);
\draw[red] (n-22) -- (n-26);
\draw[purple] (n-7) -- (n-24);
\draw[purple] (n-17) -- (n-24);
\end{tikzpicture}
    \caption{MAD pairs for $\mathbb{Z}_{31}$}
    \label{fig:MAD31}
    \end{figure}
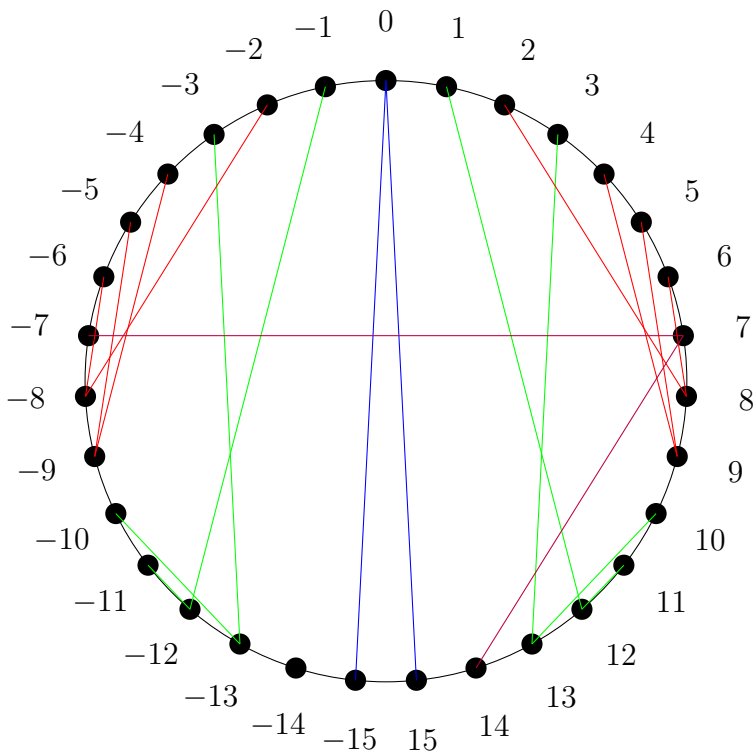

    \begin{enumerate}
        \item We have $b_1 = 6$ and $c_1 = 8$, and thus $a_1 = 2$.  Next, $b_2 = 5$ and $c_2 = 9$, so $a_2 = 4$.  This step now terminates, since we cannot have $b_3 = 4 = a_2$.
        Hence we have obtained the MAD pairs $(2,6)$ and $(4,5)$.

        \item Since $c_2 = 9$, we set $b_3 = 10$, and we have $a_3 = 3$.  Next we have $b_4 = 11$ and $a_4 = 1$.  As we have run out of candidates for $a_i$,  this step terminates and we have obtained the MAD pairs $(3, 10)$ and $(1, 11)$.

        \item Taking the negatives of the MAD pairs above, we now obtain the pairs
        \[
            (-2, -6), (-4, -5), (-3, -10), \text{ and } (-1, -11).
        \]

       \item[(3*)] We obtain the additional pair $(-7, 14)$.

       \item The last available MAD pair is $(15, -15)$.
    \end{enumerate}

    In this case, $\mu = \floor*{\frac{31}{3}}=10$ and, indeed, we have produced $10$ pairs above.  Since the sums, $a_i + b_i$, are
    \[
       0, 7, \pm 8, \pm 9, \pm 12, \pm 13,
    \]
    we see they are distinct and distinct from the constructed pairs. Therefore the algorithm produced MAD pairs.
\end{example}

Examination of the proof of Theorem \ref{theorem: MAD} results in the following.

\begin{remark} \label{remark: MAD pairs with Z}
    Let $m \in \mathbb{N}$, $m \geq 1$, and let $\mu = \floor*{\frac{m}{3}}$.
    Let $S$ be the set in Equation \eqref{S}.
    Note that all MAD pairs from Theorem \ref{theorem: MAD}, and their sums, live in $S$, and depend only on the $\mathbb{Z}$-arithmetic within $S$.  Moreover, for $m \not\equiv 6 \mod 12$, there are $\mu - 1$ additively disjoint pairs in $S \setminus \{0\}$.
\end{remark}

\begin{corollary} \label{corollary up to m/3 via MAD}
    Let $m \in \mathbb{N}$ with $m \geq 1$. If $m \not\equiv 6 \mod 12$, then $F_j$ is $\mathbb{Z}_m$-cordial for
    $0 \leq j \leq \floor*{\frac{m}{3}}$. Otherwise, $F_j$ is $\mathbb{Z}_m$-cordial for
    $0 \leq j \leq \floor*{\frac{m}{3}} - 1$.
\end{corollary}

\section{Construction to $F_{\frac{m}{2}}$}

In the proof of the theorem below, we refer to the explicit constructions of MAD pairs from Theorem \ref{theorem: MAD}, \eqref{MAD01} to \eqref{MAD8-11}, that are given in the Appendix, Section \ref{appendix}.

\begin{theorem} \label{theorem: to m/2}
    Let $m \equiv r\mod 12$ with $0 \leq r \leq 11$ and
    $r \not = 2, 4, 6, 8, 10$. Then $F_j$ is $\mathbb{Z}_m$-cordial for all $j \leq \floor*{\frac{m}{2}}.$ For  $m \equiv 2\mod 12$, $F_j$ is $\mathbb{Z}_m$-cordial for all $j \leq \floor*{\frac{m}{2}}-1.$
\end{theorem}

\begin{proof} Write $m = 12n + r$ with $0 \leq r \leq 11$
    and use the notation from Section \ref{appendix}.

    \medskip

    \textbf{Case 1:} $r = 0$ or $1$. We make use of the (MAD-1) construction for $F_{\floor*{m/3}}=F_{4n}$.

    For $r=0$, notice that each element of $\mathbb{Z}_m$ is used as an edge label exactly once.
    Other than $0$, the sums $a_i + b_i$ are not used as vertex labels. The complement of these sums in $\mathbb{Z}_m$ are the circumferential vertex labels, the $a_i$ and $b_i$, and appear as labels exactly once.

    For $r=1$, the only unused element is $-6n$. By appending the triangle with vertices $-6n$ and $3n$ with edge sum $-3n$, each element from $\mathbb{Z}_m$ is used as an edge label exactly once (except $\pm 3n$ which is used twice) and each sum $a_i+b_i$ (other than $0$ and $3n$) is not used as a vertex label.

    In either case, we can now append the following $2n-1$ triangles to $F_{4n}$ for $r=0$ and $F_{4n+1}$ for $r=1$:
    \begin{itemize}
        \item[] vertices $(3n+1)$ and $-(4n-1)$ which has edge sum $-n+2$,

        \quad\vdots
        \item[] vertices $(4n-1)$ and $-(3n+1)$ which has edge sum $n-2$,
        \item[] vertices $(5n)$ and $-(6n-1)$ which has edge sum $-n+1$,

        \quad\vdots
        \item[] vertices $(6n-1)$ and $-(5n)$ which has edge sum $n-1$.
    \end{itemize}

    For $r=1$, this yields a $\mathbb{Z}_m$-cordial labeling of $F_{6n}$ with every element in $\mathbb{Z}_{m}$ used as a vertex label exactly once.

    For $r=0$, this yields a $\mathbb{Z}_m$-cordial labeling of $F_{6n-1}$ with every element in $\mathbb{Z}_{m}$ (other than the missing $3n$) used as a vertex label exactly once. By appending the triangle with vertices $-6n$ and $3n$ with edge sum $-3n$, this can be extended to a $\mathbb{Z}_m$-cordial labeling of $F_{6n}$.

    \smallskip

    \textbf{Case 2:} $r=2$.
    In this case, we will use the revised (MAD-2) construction for $F_{\floor*{m /3}}=F_{4n}$ from Section \ref{appendix}.
    There we have two unused elements of $\mathbb{Z}_m$, $-2n+1$ and $-4n-2$.  By appending the triangle with vertices $-2n+1$ and $-4n-2$ and edge sum $-6n-1$, each element from $\mathbb{Z}_m$ is used as an edge label exactly once (except $6n+1$ which is used twice) and each sum $a_i+b_i$ (other than 0) is not used as a vertex label.

    We can now append the following $2n-1$ triangles to $F_{4n+1}$:
    \begin{itemize}
        \item[] vertices $(3n+2)$ and $-(4n+1)$ which has edge sum $-n+1$,

        \quad\vdots
        \item[] vertices $(4n+1)$ and $-(3n+2)$ which has edge sum $n-1$,
        \item[] vertices $(5n+2)$ and $-(6n)$ which has edge sum $-n+2$,

        \quad\vdots
        \item[] vertices $(6n)$ and $-(5n+2)$ which has edge sum $n-2$.
    \end{itemize}
    This yields a $\mathbb{Z}_m$-cordial labeling of $F_{6n}$ with every element in $\mathbb{Z}_{m}$ (other than the missing $6n+1$) used as a vertex label exactly once.

    \smallskip

    \textbf{Case 3:} $r=3,5,7$. In this case, we will use the (MAD-2) construction for $F_{\lfloor m/3 \rfloor}$ for $r=3,5$ and the (MAD-3) construction for $r=7$ from Section \ref{appendix}.
    For $r=3$, notice that each element from $\mathbb{Z}_m$ is used as an edge label exactly once and each sum $a_i+b_i$ (other than 0) is not used as a vertex label.

     For $r=5$, there are two unused elements, $\pm (6n+2)$. By appending the triangle with vertices $\pm (6n+2)$ and edge sum $0$, each element from $\mathbb{Z}_m$ is used as an edge label exactly once (except $0$ which is used twice) and each sum $a_i+b_i$ (other than 0) is not used as a vertex label.

     Similarly, for $r=7$ there is one unused element, $-(6n+2)$. By appending the triangle with vertices $-(6n+2)$ and $3n+1$ and edge sum $-(3n+1)$, each element from $\mathbb{Z}_m$ is used as an edge label exactly once (except  $\pm(3n+1)$ which is used twice) and each sum $a_i+b_i$ (other than 0 and $3n+1$) is not used as a vertex label.

    We can now append the following $2n$ triangles to $F_{4n+1}$ for $r=3$, $F_{4n+2}$ for $r=5$, and $F_{4n+3}$ for $r=7$:
    \begin{itemize}
        \item[] vertices $(3n+2)$ and $-(6n+1)$ which has edge sum $-3n+1$,

        \quad\vdots
        \item[] vertices $(4n+1)$ and $-(5n+2)$ which has edge sum $-n-1$,
        \item[] vertices $(5n+2)$ and $-(4n+1)$ which has edge sum $n+1$,

        \quad\vdots
        \item[] vertices $(6n+1)$ and $-(3n+2)$ which has edge sum $3n-1$.
    \end{itemize}
     This yields a $\mathbb{Z}_m$-cordial labeling of $F_{6n+1}$ for $r=3$, $F_{6n+2}$ for $r=5$, and $F_{6n+3}$ for $r=7$.

\smallskip

    \textbf{Case 4: } $r=9$ or $11$.
    In this case, we will use the (MAD-4) construction for $F_{\lfloor m/3 \rfloor}=4n+3$ from Section \ref{appendix}.
    For $r=9$, notice that each element from $\mathbb{Z}_m$ is used as an edge label exactly once and each sum $a_i+b_i$ (other than 0) is not used as a vertex label.

     For $r=11$, there are two unused elements, $\pm(6n+5)$. By appending the triangle with vertices $\pm(6n+5)$ and edge sum $0$, each element from $\mathbb{Z}_m$ is used as an edge label exactly once (except $0$ which is used twice) and each sum (other than 0) is not used as a vertex label.

     We can now append the following $2n+1$ triangles to $F_{4n+3}$ for $r=9$ and $F_{4n+4}$ for $r=11$:
    \begin{itemize}
        \item[] vertices $(3n+3)$ and $-(6n+4)$ which has edge sum $-3n-1$,

        \quad\vdots
        \item[] vertices $(4n+2)$ and $-(5n+5)$ which has edge sum $-n-3$,
        \item[] vertices $(5n+5)$ and $-(4n+3)$ which has edge sum $n+2$,

        \quad\vdots
        \item[] vertices $(6n+4)$ and $-(3n+4)$ which has edge sum $3n$,        \item[] vertices $(4n+3)$ and $-(3n+3)$ which has edge sum $n$,
    \end{itemize}
    This yields a $\mathbb{Z}_m$-cordial labeling of $F_{6n+4}$ for $r=9$ and of $F_{6n+5}$ for $r=11$.
\end{proof}

\section{Construction to $F_{\frac{2m}{3} - 1}$}

We begin with a direct construction for the cordiality of the halfway mark.

\begin{theorem} \label{theorem: m/2}
    Let $m \in \mathbb{N}$ be odd and write $m = 2k + 1$. Then
    $F_{\floor*{m/2}} = F_k$ is $\mathbb{Z}_m$-cordial.
\end{theorem}

\begin{proof}
    First look at the case of $k$ odd.
    Label the central vertex of $F_k$ with $0$. Then add $k$ triangles whose circumferential vertex labels are the pairs
    \[ (1, 2), (3, 4), \ldots, (m - 4, m - 3), (m-2, m-1). \]
    This gives $f_V(i) = 1$ for all $i$ and so is vertex-cordial.
    The above numbers also give the radial edge sums. The circumferential edge sums are $3 + 4i$ for $0 \leq i \leq k-1$. As $4$ is a unit of $\mathbb{Z}_m$, these sums are all distinct. In addition, observe that $0$ appears as the sum of the pair $(k, k + 1)$.
    In fact, it is convenient to work with $\mathbb{Z}_m$ centered about $0$ so that the circumferential edge sums are $0$ and $\pm 4i$
    for $1 \leq i \leq \frac{k - 1}{2}$.
    Therefore, $f_E(\pm 4i) = 2$ for for $1 \leq i \leq \frac{k - 1}{2}$ and $f_E(j) = 1$ for the remaining elements of $\mathbb{Z}_m$. Thus, the labeling is edge-cordial as well.

    Turn to the case of $k$ even.
    Label the central vertex of $F_k$ with $0$. Then add $k$ triangles whose circumferential vertex labels are the pairs
    \[ (1, m - 1) \,\, \text{and} \,\,
       (2, 3), (4, 5), \ldots, (m - 5, m - 4), (m-3, m-2). \]
    This gives $f_V(i) = 1$ for all $i$ and so is vertex-cordial.
    The above numbers also give the radial edge sums. The circumferential edge sums are $0$ and $5 + 4i$ for $0 \leq i \leq k-2$. As $4$ is a unit of $\mathbb{Z}_m$, the sums $5 + 4i$ are all distinct. Again, observe that $0$ also appears as the sum of the pair $(k, k + 1)$ so that the circumferential edge sums are $0$, $0$, and $\pm 4i$
    for $1 \leq i \leq \frac{k - 2}{2}$.
    Therefore, $f_E(0) = 2$, $f_E(\pm 4i) = 2$ for for $1 \leq i \leq \frac{k - 2}{2}$, and $f_E(j) = 1$ for the remaining elements of $\mathbb{Z}_m$. Thus, the labeling is edge-cordial as well.

    For use with later theorems, we also construct an alternate cordial labeling in both cases.

    For $k$ odd,
    label the central vertex of $F_k$ with $0$. Then add $k$ triangles whose circumferential vertex labels are the pairs
    \[ (1, m - 1) \,\, \text{and} \,\,
       (2, 3), (4, 5), \ldots, (m - 5, m - 4), (m-3, m-2). \]
    This gives $f_V(i) = 1$ for all $i$ and so is vertex-cordial.
    The above numbers also give the radial edge sums. The circumferential edge sums are $0$ and $5 + 4i$ for $0 \leq i \leq k-2$ and are distinct. The later sums may be rewritten as $\pm (2 + 4i)$
    for $0 \leq i \leq \frac{k - 3}{2}$.
    Therefore, $f_E(\pm (2 + 4i)) = 2$ for $0 \leq i \leq \frac{k - 3}{2}$ and $f_E(j) = 1$ for the remaining elements of $\mathbb{Z}_m$. Thus, the labeling is edge-cordial as well.

    For $k$ even,
    label the central vertex of $F_k$ with $0$. Then add $k$ triangles whose circumferential vertex labels are the pairs
    \[ (1, m - 1), (2, m - 2) \,\, \text{and} \,\,
       (3, 4), (5, 6), \ldots, (m - 6, m - 5), (m-4, m-3). \]
    This gives $f_V(i) = 1$ for all $i$ and so is vertex-cordial.
    The above numbers also give the radial edge sums. The circumferential edge sums are $0$, $0$, and $7 + 4i$ for $0 \leq i \leq k-3$, where the later sums are distinct and nonzero. These sums $7+4i$ may be rewritten as $\pm (2 + 4i)$
    for $0 \leq i \leq \frac{k - 4}{2}$.
    Therefore, $f_E(0) = 2$, $f_E(\pm (2 + 4i)) = 2$ for $0 \leq i \leq \frac{k - 4}{2}$, and $f_E(j) = 1$ for the remaining elements of $\mathbb{Z}_m$. Thus, the labeling is edge-cordial as well.
\end{proof}

We now extend cordiality up to $F_{\floor*{2m/3} - 1}$.

\begin{theorem} \label{theorem: F2m/3-1}
    Let $m \in \mathbb{N}$ be odd and write $m = 2k + 1$. Then
    $F_{j}$ is $\mathbb{Z}_m$-cordial for
    $\floor*{\frac{m}{2}} \leq j \leq
    \floor*{\frac{2m}{3}} - \varepsilon$ where $\varepsilon = 1$ when $k$ is even and $3 \nmid k$ and $\varepsilon = 0$ otherwise.
\end{theorem}

\begin{proof}
    Begin with $k$ odd and the second construction of $F_k$ given in the proof of Theorem \ref{theorem: m/2}. Recall that construction had a uniform vertex distribution of $f_V(i) = 1$ for all $i$. The edge distribution was
    $f_E(\pm (2 + 4i)) = 2$ for $0 \leq i \leq \frac{k - 3}{2}$ and $f_E(j) = 1$ for the remaining elements of $\mathbb{Z}_m$.

    As a result, we may preserve cordiality by adding triangles with circumferential vertex labels $(a_r, b_r)$ as long as, cumulatively, $a_r, b_r, a_r + b_r$ are distinct and not of the form $\pm (2 + 4i)$ for $0 \leq i \leq \frac{k - 3}{2}$.
    The subset of $\mathbb{Z}_m$ of all elements distinct from $\pm (2 + 4i)$ for $0 \leq i \leq \frac{k - 3}{2}$ is given by
    \[ \{ -m - 1, -m + 3, \ldots, -8, -4, 0, 4, 8, \ldots, m - 3, m + 1 \} \]
    which can be written as $4\cdot S$ where $S$ is the set defined by
    \[ \left\{ -\frac{m + 1}{4}, -\frac{m + 1}{4} + 1, \ldots, -2, -1, 0,
              1, 2, \ldots, \frac{m + 1}{4} - 1, \frac{m + 1}{4}\right\}. \]
    Working with $\mathbb{Z}_{(m + 3)/2}$ in Remark \ref{remark: MAD pairs with Z} (and multiplying by $4$ at the end), we observe that we may add an additional $\floor*{\frac{m + 3}{6}}$ cordial triangles to $F_k$. (Note here that $\frac{m+3}2 \equiv 6 \mod 12$ implies $m \equiv 9 \mod 24$ which is impossible!) Writing
    $k = 2s + 1$ so that $m = 4s + 3$, we get
    \begin{align*}
        k + \floor*{\frac{m + 3}{6}}
        &=  2s + 1 + \floor*{\frac{2s + 3}{3}} \\
        &= \floor*{\frac{8s + 6}{3}}
        = \floor*{\frac{2m}{3}}
    \end{align*}
    as desired.

    Turn now to the case of
    $k$ even and use the second construction of $F_k$ given in the proof of Theorem \ref{theorem: m/2}. Recall that construction had a uniform vertex distribution of $f_V(i) = 1$ for all $i$. The edge distribution was $f_E(0) = 2$,
    $f_E(\pm (2 + 4i)) = 2$ for $0 \leq i \leq \frac{k - 4}{2}$, and $f_E(j) = 1$ for the remaining elements of $\mathbb{Z}_m$.

    As a result, we may preserve cordiality by adding triangles with circumferential vertex labels $(a_r, b_r)$ as long as, cumulatively, $a_r, b_r, a_r + b_r$ are distinct and not of the form $0$ or $\pm (2 + 4i)$ for $0 \leq i \leq \frac{k - 4}{2}$.
    The subset of $\mathbb{Z}_m$ of all elements distinct from $0$ and $\pm (2 + 4i)$ for $0 \leq i \leq \frac{k - 4}{2}$ is given by
    \[ \{ -m - 3, -m + 1, \ldots, -8, -4, \,\, 4, 8, \ldots, m - 1, m + 3 \} \]
    which can be written as $4\cdot S$ where $S$ is the set defined by
    \[ \left\{ -\frac{m + 3}{4}, -\frac{m + 3}{4} + 1, \ldots, -2, -1, \,\,
              1, 2, \ldots, \frac{m + 3}{4} - 1, \frac{m + 3}{4}\right\}. \]
    Working with $\mathbb{Z}_{(m + 5)/2}$ in Remark \ref{remark: MAD pairs with Z}, we observe that we may add an additional $\floor*{\frac{m + 5}{6}} - 1$ cordial triangles to $F_k$. (Note here that $\frac{m+5}2 \equiv 6 \mod 12$ implies $m \equiv 7 \mod 24$ which is impossible!) Writing $k = 2s$ so that
    $m = 4s + 1$, we get
    \begin{align*}
        k + \floor*{\frac{m + 5}{6}} - 1
        &=  2s + \floor*{\frac{2s + 3}{3}} - 1 \\
        &= \floor*{\frac{8s + 3}{3}} - 1
        \geq  \floor*{\frac{8s + 2}{3}} - 1
        = \floor*{\frac{2m}{3}} - 1.
    \end{align*}
    More precisely, $k + \floor*{\frac{m + 5}{6}} - 1$ is equal to $\floor*{\frac{2m}{3}}$ exactly when $3 \mid s$ and, otherwise, when
    $m = 12q + 5$ or $m = 12q + 9$, is one less.
\end{proof}

\section{Construction to $F_m$ for $3 \mid m$}

The following defines a shift of the labels of just
the circumferential vertices of $F_n$.
\begin{definition}
    Let $m, n \in \mathbb{N}$ with $m \geq 1$ and
    $\ell: V \rightarrow \mathbb{Z}_m$ be a labeling
    of the vertices of $F_n$. For $i_0 \in \mathbb{Z}_m$,
    write $F_n^{i_0}$ for $F_n$ equipped with a new
    vertex labeling, $\ell'$, given by
    \[
        \ell'(v) =
        \begin{cases}
            \ell(v), \text{\,\, if $v$ is the central vertex}\\
            \ell(v) + i_0, \text{\,\, otherwise}
        \end{cases}
    \]
    for $v \in V$.
\end{definition}

Note that $F_n^{i_0}$ shifts the radial edges by $i_0$ and the circumferential edges by $2i_0$.

\begin{theorem} \label{theorem: 2m/3 to m for 3 div m}
    Let $m \in \mathbb{N}$ be odd and divisible by $3$.
    Then $F_j$ is $\mathbb{Z}_m$-cordial for
    $\frac{2m}{3} \leq j \leq m$.
\end{theorem}

\begin{proof}
    Write $m = 3(2k + 1)$ so that $\frac{m}{3} = 2k + 1$.
    We begin with an alternate construction of $F_{2k + 1}$ that depends on $3 \mid m$. Label the
    central vertex with $0$. Then add $2k + 1$
    circumferential triangles with labels
    \[ (m - 1,1), \,\, (2, 4), (5, 7), \ldots,
        (m - 4, m - 2). \]
    The sums are then
    $6i$ for $0 \leq i \leq 2k$.
    The distribution is
    $f_V(3i) = 0$ for $0 < 3i < m$ with $f_V(i) = 1$
    otherwise and $f_E(i) = 1$ for all $i$.

    Consider now $F_{2k + 1}^1 + F_{2k + 1}^{-1}$. This
    labeling has $\frac{2m}{3} = 4k + 2$ triangles, but
    is not quite cordial. However, we will get a
    $\mathbb{Z}_m$-cordial labeling of $F_{4k + 2}$ if
    we remove the triangles with circumferential vertices
    $(0,\pm 2)$ and replace them with \emph{two} copies of
    $(2, -2)$. The distribution of this construction is
    $f_V(\pm 2) = 2$, $f_V(3i) = 2$ for $0 < 3i < m$, and $f_V(i) = 1$ otherwise and $f_E(i) = 2$ for all $i$.

    We may now add to $F_{4k + 2}$ the triangles from $F_{2k + 1}$, one at a time, except for $(2, 4)$ and $(-2,-4)$. This preserves cordiality and finishes the proof for $F_j$
    with $\frac{2m}{3} \leq j \leq m - 2$. It only
    remains to construct $F_{m - 1}$. This can be done
    by taking the cordial construction of $F_m$ from Theorem
    \ref{theorem: Fm or F2m reduction}(a) and
    removing the triangles with circumferential vertices
    $(0,\pm 1)$ and replacing them with the triangle $(1, -1)$.
\end{proof}

Combining Theorems \ref{theorem: Fm or F2m reduction}, \ref{theorem: to m/2}, \ref{theorem: F2m/3-1},
\ref{theorem: 2m/3 to m for 3 div m} and Corollary \ref{corollary up to m/3 via MAD}, we establish Conjecture \ref{conjecture for cordiality of friendship} for the case of $m$ odd with $3 \mid m$.

\begin{corollary} \label{corollary done odd m div by 3}
        Let $m, n \in \mathbb{N}$ with $m$ odd and divisible by $3$.
        Then $F_n$ is $\mathbb{Z}_m$-cordial for all $n$.
\end{corollary}

\section{Closing Remarks}

A proof of Conjecture \ref{conjecture: weakly cordial}, that trees are
weakly $A$-cordial for any abelian group $A$, seems far off at this time. However, further partial results would be very interesting.

It would also be very interesting to verify Conjecture \ref{conjecture for cordiality of friendship} in general, that is, to show that
$F_n$ is always $\mathbb{Z}_m$-cordial except when
\[ 2 \mid n, \,\, 4 \nmid n, \,\, m = \frac{3n}{d} \]
for $d$ a positive, odd divisor of $3n$.

For $m$ odd, ad hoc arguments can be given to show $F_n$ is $\mathbb{Z}_m$-cordial
in numerous other cases. For example,
for $m = 8k + 1$, the missing cordiality from $F_m$ down to
$F_{\floor*{2m/3}}$ can be constructed with the following techniques: Begin by removing $(0,\pm 1)$ and replacing those triangles with $(1, -1)$ in the construction from Theorem \ref{theorem: Fm or F2m reduction}(a).
After that remove the roughly $\frac{m}{4}$ triangles
of the form $(2 + 4i, 3 + 4i)$. After that, remove
triangles $(8+16i, 9+16i)$. Finally, remove additively disjoint
pairs from the remaining vertices labeled with multiples of $16$.

It is worth noting that Theorem \ref{theorem: Fm or F2m reduction} and Corollary \ref{corollary up to m/3 via MAD} allow computer verification of $\mathbb{Z}_m$-cordiality of $F_n$ for all $n$ when $m$ is not very large. At the very least, this would allow verification of Conjecture \ref{conjecture for cordiality of friendship} for many $m$.

For $m$ even, prohibited cordiality
cases from Theorem \ref{theorem: obstruction} come into play and, presumably,
increase the complexity. For example, it is fairly straightforward to
inductively show that
$F_n$ is $\mathbb{Z}_2$-cordial if and only if $n$ is odd or
if $n$ is (even and) divisible by $4$. However, when $4 \mid m$, the cases are expected to be easier. For example, using techniques from this paper, it is straightforward
to show that $F_n$ is $\mathbb{Z}_4$-cordial
for all $n$. In general, the case of $m$ even
seems more difficult than the case of $m$ odd.

\section{Appendix: MAD Details} \label{appendix}

Here we give the explicit MAD pairs and their sums that are constructed from the algorithm in Theorem \ref{theorem: MAD}.  For the sake of exposition, we have re-indexed the $a_i$ and $b_i$. The results depend upon the residue class of $m$ modulo $12$ and so we write $m = 12n + r$, with $0 \leq r \leq 11$. There are four main cases:

\medskip

\noindent \textbf{Case 1:} $r=0$ or $1$.  In this case $\mu = 4n$.

\begin{equation}
\label{MAD01}
    \tag{MAD-1}
    (a_i, b_i) =
    \begin{cases}
    (2n-2i, \: 2n+i), & 1 \leq i < n,\\
    (2i-2n+1, \: 6n-i-1), & n \leq i < 2n,\\
    (-a_{i-2n+1}, \: -b_{i-2n+1}), & 2n \leq i \leq 4n-2,\\
    (-3n, 6n), & i = 4n-1,\\
    (-2n, 2n), & i = 4n.
    \end{cases}
\end{equation}
This yields the following sets in which each element appears exactly once (with ellipses denoting consecutive integers):
\begin{align*}
    \{a_i\} &= \pm\{1, \ldots, 2n-1\} \cup \{-2n, -3n\},\\
    \{b_i\} &= \pm\{2n+1,\ldots,3n-1, 4n, \ldots, 5n-1\} \cup \{2n, 6n\},\\
    \{a_i + b_i\} &= \pm\{0, 3n+1,\ldots,4n-1,5n,\ldots,6n-1\} \cup \{3n\}.
\end{align*}

\noindent \textbf{Case 2:} $r=3$, $4$, or $5$.  [See below for $r=2$.]  In this case $\mu = 4n+1$.

\begin{equation}
\label{MAD2345}
    \tag{MAD-2}
    (a_i, b_i) =
    \begin{cases}
    (2n - 2i +2, \: 2n+i), & 1 \leq i \leq n,\\
    (2i-2n-1, \: 6n-i+2), & n < i \leq 2n,\\
    (-a_{i-2n}, \: -b_{i-2n}), & 2n < i \leq 4n,\\
    (-3n-1, \: 3n+1), & i = 4n+1.
    \end{cases}
\end{equation}
We have the following:
\begin{align*}
    \{a_i\} &= \pm\{1, \ldots, 2n\} \cup \{-3n-1\},\\
    \{b_i\} &= \pm\{2n+1,\ldots, 3n,4n+2, \ldots, 5n+1\} \cup \{3n+1\},\\
    \{a_i + b_i\} &= \pm\{ 0, 3n+2,\ldots,4n+1,5n+2,\ldots,6n+1\}.
\end{align*}

When $r=2$, we have $\mu = 4n$, and we modify the construction \eqref{MAD2345} by removing the pair $(a_{4n}, b_{4n}) = (-2n+1, \: -4n-2)$.  This is because in the case $r=2$, we have $a_{2n} + b_{2n} = a_{4n}+b_{4n} = 6n+1$.

\bigskip

\noindent\textbf{Case 3:} $r=7$. [See below for $r=6$.] In this case $\mu = 4n+2$.

\begin{equation}
    \label{MAD7}
    \tag{MAD-3}
    (a_i, b_i) =
    \begin{cases}
     (2n - 2i +2, \: 2n+i), & 1 \leq i \leq n,\\
    (2i-2n-1, \: 6n-i+2), & n < i \leq 2n,\\
    (-a_{i-2n}, \: -b_{i-2n}), & 2n < i \leq 4n,\\
    (-3n-1, \: 6n+2), & i = 4n+1,\\
    (-6n-3,\: 6n+3), & i = 4n+2.
    \end{cases}
\end{equation}
We have the following:
\begin{align*}
    \{a_i\} &= \pm\{1,\ldots,2n\} \cup \{-3n-1, -6n-3\},\\
    \{b_i\} &= \pm\{2n+1,\ldots,3n,4n+2,\ldots,5n+1\}\cup\{6n+2,\: 6n+3\},\\
    \{a_i + b_i\} &= \pm\{ 0, 3n+2,\ldots,4n+1,5n+2,\ldots,6n+1\} \cup \{3n+1\}.
\end{align*}

In the case $r=6$ (where $\mathbb{Z}_m$ is not MAD), we remove the final pair $(a_{4n+2},b_{4n+2})$ from the construction \eqref{MAD7}.

\bigskip

\noindent\textbf{Case 4:} $r=9$, $10$, or $11$.  [See below for $r=8$.]  In this case $\mu = 4n+3$.

\begin{equation}
    \label{MAD8-11}
    \tag{MAD-4}
    (a_i, b_i) =
    \begin{cases}
        (2n-2i+3, \: 2n+i+1), & 1 \leq i \leq n+1,\\
        (2i - 2n - 2, \: 6n -i + 5), & n+2 \leq i \leq 2n+1,\\
        (-a_{i-2n-1}, \: -b_{i-2n-1}), & 2n+2 \leq i \leq 4n+2,\\
        (-5n-4, \: 5n+4), & i = 4n+3.
    \end{cases}
\end{equation}
We have the following:
\begin{align*}
    \{a_i\} &= \pm\{1, \ldots, 2n+1\} \cup \{-5n-4\},\\
    \{b_i\} &= \pm\{2n+2, \ldots, 3n+2,4n+4, \ldots, 5n+3\} \cup \{5n+4\},\\
    \{a_i + b_i\} &= \pm\{ 0, 3n+3,\ldots,4n+3,5n+5,\ldots,6n+4\}.
\end{align*}

When $r=8$, we have $\mu = 4n+2$, and we modify the construction \eqref{MAD8-11} by removing the pair $(a_{4n+2},\:b_{4n+2}) = (-2n, -4n-4)$.  This is because in the case $r=8$, we have $a_{2n+1} + b_{2n+1} = a_{4n+2} + b_{4n+2} = 6n+4$.


\newpage

\bibliographystyle{alpha}
\bibliography{refs}

\begin{thebibliography}{TWW19}

\bibitem[Cah87]{Cahit}
I.~Cahit.
\newblock Cordial graphs: a weaker version of graceful and harmonious graphs.
\newblock {\em Ars Combin.}, 23:201--207, 1987.

\bibitem[CE00]{NPComplete}
N.~Cairnie and K.~Edwards.
\newblock The computational complexity of cordial and equitable labelling.
\newblock {\em Discrete Math.}, {\bf 216}:29--34, 2000.

\bibitem[CG21]{CichaczGorlich21}
S.~Cichacz and A.~G\"{o}rlich.
\newblock $\mathbb{Z}_2 \times \mathbb{Z}_2$-cordial cycle-free hypergraphs.
\newblock {\em Discuss. Math. Graph Theory}, {\bf 41}:1021--1040, 2021.

\bibitem[CGT13]{CichaczGorlichTuza13}
S.~Cichacz, A.~G\"orlich, and Z.~Tuza.
\newblock Cordial labeling of hypertrees.
\newblock {\em Discrete Math.}, {\bf 313}:2518--2524, 2013.

\bibitem[Cic22]{Cichacz22}
S.~Cichacz.
\newblock On some graph-cordial abelian groups.
\newblock {\em Discrete Math.}, {\bf 345}:112815, 7pp., 2022.

\bibitem[DKN17]{treesAre6Cordial}
K.~Driscoll, E.~Krop, and M.~Nguyen.
\newblock All trees are six-cordial.
\newblock {\em Electron. J. Graph Theory Appl.}, {\bf 5}:21--35, 2017.

\bibitem[Dri19]{treesAre7Cordial}
K.~Driscoll.
\newblock All trees are seven-cordial.
\newblock {\em {\rm arXiv:1909.12351}}, 2019.

\bibitem[ERS66]{ErdosRenyiSos}
P.~Erd\H{o}s, A.~R\'{e}nyi, and V.T. S\'{o}s.
\newblock On a problem of graph theory.
\newblock {\em Studia Sci. Math. Hungar.}, {\bf 1}:215--235, 1966.

\bibitem[Gal98]{gallian}
J.A. Gallian.
\newblock A dynamic survey of graph labeling.
\newblock {\em Electron. J. Combin.}, {\bf 5}:1--47, 1998.

\bibitem[Har69]{harary}
F.~Harary.
\newblock {\em Graph Theory}.
\newblock Addison-Wesley Publishing Co., Reading, Mass.-Menlo Park,
  Calif.-London, 1969.

\bibitem[HHOr04]{Haanpaa04}
H.~Haanp\"{a}\"{a}, A.~Huima, and P.~\"{O}sterg\aa rd.
\newblock Sets in {$\mathbb{Z}_n$} with distinct sums of pairs.
\newblock {\em Discrete Appl. Math.}, {\bf 138}:99--106, 2004.
\newblock Optimal discrete structures and algorithms (ODSA 2000).

\bibitem[HOr07]{Haanpaa07}
H.~Haanp\"{a}\"{a} and P.~\"{O}sterg\aa rd.
\newblock Sets in abelian groups with distinct sums of pairs.
\newblock {\em J. Number Theory}, {\bf 123}:144--153, 2007.

\bibitem[Hov91]{hovey}
M.~Hovey.
\newblock {$A$}-cordial graphs.
\newblock {\em Discrete Math.}, {\bf 93}:183--194, 1991.

\bibitem[Kan14]{friendshipsAre7Cordial}
K.~Kanani.
\newblock 7-cordial labeling of some standard graphs.
\newblock {\em Int. J. Appl. Math. Res.}, {\bf 3}:547--560, 2014.

\bibitem[KM15]{friendshipsAre5Cordial}
K.~Kanani and M.~Modha.
\newblock 5-cordial labeling of some standard graphs.
\newblock In {\em Proceedings of 8th National Level Science Symposium},
  volume~{\bf 2}, pages 43--48, 2015.

\bibitem[KR14]{friendshipsAre4Cordial}
K.~Kanani and N.~Rathod.
\newblock Some new 4-cordial graphs.
\newblock {\em J. Math. Comp. Sci.}, {\bf 4}:834--848, 2014.

\bibitem[PP21]{patriasPechenik}
R.~Patrias and O.~Pechenik.
\newblock Path-cordial abelian groups.
\newblock {\em Australas. J. Combin.}, {\bf 80}:157--166, 2021.

\bibitem[PW12]{pechenikWise}
O.~Pechenik and J.~Wise.
\newblock Generalized graph cordiality.
\newblock {\em Discuss. Math. Graph Theory}, {\bf 32}:557--567, 2012.

\bibitem[Ris07]{riskin}
A.~Riskin.
\newblock $\mathbb{Z}_{2}^{2}$-cordiality of complete and complete bipartite
  graphs.
\newblock {\em {\rm arXiv:0709.0290}}, 2007.

\bibitem[TWW19]{TuczynskiWenusKesek19}
M.~Tuczyński, P.~Wenus, and K.~Wesek.
\newblock On cordial labeling of hypertrees.
\newblock {\em Discrete Math. Theor. Comput. Sci.}, {\bf 21}:\#5658, 14pp.,
  2019.

\end{thebibliography}

\end{document}